\documentclass[b5paper]{report}
\usepackage[italian]{babel}
\usepackage[utf8]{inputenc}
\usepackage{amsmath}
\usepackage{amssymb}
\usepackage{amsthm}
\usepackage{enumerate}
\usepackage{xspace}
\usepackage{bm}
\usepackage{euscript}
\usepackage{cancel}
\usepackage{mathrsfs}
\usepackage{pstricks}
\usepackage{multido}
\usepackage{pst-plot}
\usepackage{pst-node}
\usepackage{pst-3d}
\makeatletter
\def\pst@picture{%
  \@ifnextchar[%
  {\pst@picture@}%
  {\pst@picture@[shift=0]}}
\def\pst@picture@[#1]#2(#3,#4){%
  \@ifnextchar(%
  {\pst@@picture@[#1](#3,#4)}
  {\pst@@picture@[#1](0,0)(#3,#4)}}
\def\pst@@picture@[#1](#2,#3)(#4,#5){%
  \pssetylength\pst@dimb{#3}%
  \pssetylength\pst@dimd{#5}%
  \ifdim\pst@dimb>\pst@dimd%
  \pst@dimg=\pst@dimb%
  \pst@dimb=\pst@dimd%
  \pst@dimd=\pst@dimg%
  \fi%
  \pst@@@picture@#1,(#2,\pst@dimb)(#4,\pst@dimd)}
\def\pst@@@picture@{%
  \@ifnextchar,%
  {\pst@@@@picture@ shift=\pst@dimb}%
  {\pst@@@@picture@}}
\def\pst@@@@picture@#1,{%
  \set@ps#1==\@nil%
  \@ifnextchar(%
  {\pst@@@picture[]}%
  {\pst@@@@picture@}}
\def\set@ps#1#2=#3=#4\@nil{%
  \ifx#4=%
  \psset{#1#2=#3}%
  \else%
  \ifx#1b%
  \pst@dimg=0pt%
  \else%
  \ifx#1o%
  \pst@dimg=\pst@dimb%
  \else%
  \pst@dimg=-\pst@dimd%
  \advance\pst@dimg\pst@dimb%
  \ifx#1t%
  \else%
  \ifx#1c%
  \pst@dimg=0.5\pst@dimg%
  \else%
  \pst@dimg=#1#2\pst@dimg%
  \fi%
  \fi%
  \fi%
  \fi%
  \psset{shift=\pst@dimg}%
  \fi}
\newlength\fig@xput
\newlength\fig@yput
\newlength\fig@height
\newlength\fig@width
\newlength\fig@rightleft
\newlength\fig@updown
\newlength\fig@distance
\newlength\fig@temp
\newlength\fig@textright
\newlength\fig@textleft
\newlength\fig@framesep
\newcounter{fig@counta}
\newcounter{fig@countb}
\def\@fig@rightleft#1{\pssetylength\fig@rightleft{#1}}%
\def\@fig@updown#1{\pssetxlength\fig@updown{#1}}%
\def\@fig@space#1{\pssetxlength\fig@temp{#1}\advance\fig@distance\fig@temp}%
\def\@fig@lines#1{\setcounter{fig@counta}{#1}}%
\def\@fig@fillcolor#1{\def\fill@color{#1}}%
\def\@fig@fillstyle#1{\def\fill@style{#1}}%
\def\@fig@framesep#1{\pssetxlength\fig@framesep{#1}}%
\def\pspicinsq{%
  \@ifnextchar[{\@pspicinsq}{\@@pspicinsq}}%
\def\@pspicinsq[#1]#2{\pspicins[#1]{\quadro{#2}}}%
\def\@@pspicinsq#1{\pspicins{\quadro{#1}}}%
\def\pspicinsw{%
  \@ifnextchar[{\@pspicinsw,}{\@pspicinsw[]}}
\def\@pspicinsw#1[#2]{%
  \pspicins[fillstyle=none,fillcolor=white,framesep=0pt#1#2]}
\def\pspicins{%
  \psset{unit=1pt}%
  \def\fill@color{Bluartico}%
  \def\fill@style{solid}%
  \fig@framesep=2pt%
  \fig@rightleft=0pt%
  \fig@distance=3\baselineskip%
  \divide\fig@distance7%
  \fig@updown=0pt%
  \setcounter{fig@counta}{0}%
  \@ifnextchar[{\@pspicins}{\@@pspicins;}}%
\def\@pspicins[#1]{%
  \get@pspicins{\@@pspicins}#1,;}
\def\get@pspicins#1#2=#3,{%
  \csname @fig@#2\endcsname{#3}%
  \@ifnextchar;{#1}{\get@pspicins{#1}}}%
\def\@@pspicins;#1{%
  \settoheight\fig@height{\sfondo{#1}}
  \settodepth\fig@width{\sfondo{#1}}
  \advance\fig@height\fig@width
  \settowidth\fig@width{\sfondo{#1}}%
  \advance\fig@width-\fig@rightleft%
  \count@lines{\fig@height}%
  \fig@textleft=\textwidth%
  \advance\fig@textleft-\linewidth%
  \fig@textright=\linewidth%
  \advance\fig@textright-\fig@width%
  \advance\fig@textright-\fig@distance%
  \fig@xput=\linewidth%
  \advance\fig@xput\fig@rightleft%
  \fig@yput=\baselineskip%
  \fig@yput=-\thefig@counta\fig@yput%
  \advance\fig@yput\baselineskip%
  \advance\fig@yput\fig@updown%
  \fig@parshape{\the\fig@textleft}{\the\fig@textright}{\the\fig@textleft}{\the\linewidth}%
  \rput[rb](\pst@number\fig@xput,\pst@number\fig@yput){%
    \psset{fillcolor=\fill@color}\sfondo{#1}}\ignorespaces}
\def\count@lines#1{%
  \setlength\fig@temp#1%
  \count@lines@}%
\def\count@lines@{%
  \ifdim\fig@temp>0pt%
  \stepcounter{fig@counta}%
  \advance\fig@temp-\baselineskip%
  \def\next{\count@lines@}%
  \else%
  \def\next{}%
  \fi%
  \next}%
\def\fig@parshape#1#2#3#4{%
  \stepcounter{fig@counta}%
  \setcounter{fig@countb}{\thefig@counta}%
  \fig@parshape@{\thefig@countb}{#1 #2}{#3 #4}}%
\def\fig@parshape@#1#2#3{%
  \ifnum\thefig@counta=1%
  \def\next{\parshape #1 #3}%
  \else%
  \addtocounter{fig@counta}{-1}%
  \def\next{\fig@parshape@{#1}{#2}{#2 #3}}%
  \fi%
  \next}%
\def\nopic{\parshape1 \the\fig@textleft \the\linewidth}%
\newpsobject{sfondo}{psframebox}{linestyle=none,fillstyle=\fill@style,framesep=\the\fig@framesep}%
\def\psparshape{%
  \@ifnextchar[{\ps@parshape}{\ps@parshape[0pt]}}
\def\ps@parshape[#1]#2{%
  \fig@xput=\fig@textright%
  \fig@yput=\fig@textleft%
  \advance\fig@xput-#1%
  \advance\fig@yput#1%
  \setcounter{fig@counta}{#2}%
  \fig@parshape{\the\fig@yput}{\the\fig@xput}{\the\fig@yput}{\the\linewidth}}
\def\pspiccenterw{%
  \fig@framesep=0pt%
  \def\fill@color{white}%
  \def\fill@style{none}%
  \@ifnextchar[{\ps@piccenter}{\ps@@piccenter;}}%
\def\pspiccenter{%
  \fig@framesep=2pt%
  \def\fill@color{Bluartico}%
  \def\fill@style{solid}%
  \@ifnextchar[{\ps@piccenter}{\ps@@piccenter;}}%
\def\ps@piccenter[#1]{%
  \get@pspicins{\ps@@piccenter}#1,;}%
\def\ps@@piccenter;#1{%
  {\psset{fillcolor=\fill@color}%
    \makebox[\the\linewidth][c]{\sfondo{#1}}}}%
  \def\thelast#1{\makebox[\the\fig@temp][s]{#1}}

\makeatother
\pstheader{pst-vect.pro}
\makeatletter
\begingroup
\catcode`\|=13
\catcode`\;=13
\catcode`\+=13
\catcode`\*=13
\catcode`\?=13
\catcode`\>=13
\catcode`\<=13
\catcode`\!=13
\catcode`\/=13
\catcode`\`=13
\gdef\pst@activemathcoor{%
  \def|{\string|}%
  \def;{\string;}%
  \def+{\string+}%
  \def*{\string*}%
  \def?{\string?}%
  \def>{\string>}%
  \def<{\string<}%
  \def!{\string!}%
  \def/{\string/}%
  \def`{\string`}}
\endgroup
\def\MathCoor{%
  \SpecialCoor
  \def\pst@@getcoor##1{%
    \begingroup
    \pst@activemathcoor
    \xdef\pst@tempg{##1}%
    \endgroup
    \expandafter\if@sum@coor\pst@tempg+;\@nil}}
\def\if@sum@coor#1+{%
  \@ifnextchar;{\if@mul@coor#1*}{\sum@coor#1+}}
\def\sum@coor#1+#2\@nil{%
  \begingroup
  \if@sum@coor#2\@nil
  \let\pst@sumx\pst@coor
  \if@mul@coor#1*;\@nil
  \xdef\pst@tempg{\pst@sumx exch \pst@coor 3 1 roll add 3 1 roll add }%
  \endgroup
  \let\pst@coor\pst@tempg}
\def\if@mul@coor#1*{%
  \@ifnextchar;{\if@rotate@coor#1?}{\@mul@coor{#1}}}
\def\@mul@coor#1#2*;\@nil{%
  \begingroup
  \isit@modulo@coor#1]\@nil
  \if@sum@coor#2+;\@nil
  \xdef\pst@tempg{\pst@coor \vect@number dup 4 3 roll mul 3 1 roll mul }%
  \endgroup
  \let\pst@coor\pst@tempg}
\def\if@rotate@coor#1?{%
  \@ifnextchar;{\if@vect@coor#1>}{\rotate@coor{#1}}}
\def\rotate@coor#1#2?;\@nil{%
  \begingroup
  \isit@angle@coor#1]\@nil
  \if@sum@coor#2+;\@nil
  \xdef\pst@tempg{%
    \pst@coor \vect@number dup sin exch cos 4 copy 4 3 roll mul 
    3 1 roll mul sub 5 1 roll 3 2 roll mul 3 1 roll mul add }%
  \endgroup
  \let\pst@coor\pst@tempg}
\def\isit@modulo@coor{%
  \@ifnextchar`{\modulo@coor}{\number@coor}}
\def\isit@angle@coor{%
  \@ifnextchar`{\angle@coor}{\number@coor}}
\def\angle@coor`{\@ifnextchar`{\angle@neg@coor}{\angle@@coor}}
\def\angle@neg@coor`#1]\@nil{%
  \if@sum@coor#1+;\@nil
  \edef\vect@number{\pst@coor exch 2 copy dup 
    mul exch dup mul add 0 eq {pop} {atan} ifelse neg }}
\def\angle@@coor#1]\@nil{%
  \if@sum@coor#1+;\@nil
  \edef\vect@number{\pst@coor exch 2 copy dup 
    mul exch dup mul add 0 eq {pop} {atan} ifelse }}
\def\modulo@coor`{\@ifnextchar`{\modulo@div@coor}{\modulo@@coor}}
\def\modulo@div@coor`#1]\@nil{%
  \if@sum@coor#1+;\@nil
  \edef\vect@number{1 \pst@coor \back@unit dup 
    mul exch dup mul add sqrt dup 0 eq {pop} {div} ifelse }}
\def\modulo@@coor#1]\@nil{%
  \if@sum@coor#1+;\@nil
  \edef\vect@number{\pst@coor \back@unit dup mul exch dup mul add sqrt }}
\def\number@coor#1]\@nil{%
  \edef\vect@number{#1 }}
\def\if@vect@coor#1>{%
  \@ifnextchar;{\if@mixed@coor#1|}{\vect@coor{#1}}}
\def\vect@coor#1#2>;\@nil{%
  \begingroup
  \if@sum@coor#2+;\@nil
  \let\pst@vectx\pst@coor
  \if@sum@coor#1+;\@nil
  \xdef\pst@tempg{\pst@vectx exch \pst@coor 3 1 roll sub 3 1 roll sub }%
  \endgroup
  \let\pst@coor\pst@tempg}
\def\if@mixed@coor#1|{%
  \@ifnextchar;{\if@versus@coor#1;}{\math@mixed@coor{#1}}}
\def\math@mixed@coor#1#2|;\@nil{%
  \begingroup
  \if@sum@coor#1+;\@nil
  \let\pst@tempa\pst@coor
  \if@sum@coor#2+;\@nil
  \xdef\pst@tempg{\pst@tempa pop \pst@coor exch pop }%
  \endgroup
  \let\pst@coor\pst@tempg}
\def\if@versus@coor{%
  \@ifnextchar/{\versus@coor}{\if@raw@coor}}
\def\versus@coor/#1;;\@nil{%
  \begingroup
  \if@sum@coor#1+;\@nil
  \xdef\pst@tempg{\pst@coor exch 2 copy dup mul exch dup mul 
    add dup 0 eq {pop 1} if sqrt dup 3 1 roll div 3 1 roll div \step@unit }%
  \endgroup
  \let\pst@coor\pst@tempg}
\def\if@raw@coor{%
  \@ifnextchar!{\math@raw@coor}{\if@polar@coor}}
\def\math@raw@coor!#1;;\@nil{%
  \edef\pst@coor{#1 \step@unit }}
\def\if@polar@coor#1;{%
  \@ifnextchar;{\if@cartesian@coor#1,}{\math@polar@coor{#1}}}
\def\math@polar@coor#1#2;;\@nil{%
  \edef\pst@coor{#2 cos #1 mul #2 sin #1 mul \step@unit }}
\def\if@cartesian@coor#1,{%
  \@ifnextchar;{\isit@vector@coor#1;}{\math@cartesian@coor{#1}}}
\def\math@cartesian@coor#1#2,;\@nil{%
  \pssetxlength\pst@dimg{#1}%
  \pssetylength\pst@dimh{#2}%
  \edef\pst@coor{\pst@number\pst@dimg \pst@number\pst@dimh}}%
\def\isit@vector@coor{%
  \@ifnextchar<{\what@node@coor}{\vector@coor}}
\def\what@node@coor<{%
  \@ifnextchar[{\Node@coor}{\node@coor}}
\def\vector@coor#1;;\@nil{%
  \edef\pst@coor{V@#1 \step@unit }}
\def\vnode{\@ifnextchar({\vnode@}{\vnode@(0,0)}}
\def\vnode@(#1)#2{%
  \pst@getvect{#1}\pst@vecta%
  \pnode@(!/V@#2 {\pst@vecta} bind def V@#2){#2}}
\def\back@unit{\pst@number\psyunit div exch \pst@number\psxunit div exch }
\def\step@unit{\pst@number\psyunit mul exch \pst@number\psxunit mul exch }
\def\pst@getvect#1#2{%
  \begingroup
  \pst@@getcoor{#1}%
  \xdef\pst@tempg{\pst@coor \back@unit }%
  \endgroup
  \let#2\pst@tempg}
\def\intersection(#1)(#2)(#3){\@ifnextchar({\@inters(#1)(#2)(#3)}{\@inters(0,0)(#1)(#2)(#3)}}
\def\@inters(#1)(#2)(#3)(#4)#5{%
  \pst@getvect{#1}\pst@vecta%
  \pst@getvect{#2}\pst@vectb%
  \pst@getvect{#3}\pst@vectc%
  \pst@getvect{#4}\pst@vectd%
  \vnode(!\pst@vecta \pst@vectb \pst@vectc \pst@vectd inters){#5}}
\def\projection(#1)(#2){\@ifnextchar({\@projec(#1)(#2)}{\@projec(0,0)(#1)(#2)}}
\def\@projec(#1)(#2)(#3)#4{%
  \pst@getvect{#1}\pst@vecta%
  \pst@getvect{#2}\pst@vectb%
  \pst@getvect{#3}\pst@vectc%
  \vnode(!\pst@vecta \pst@vectb \pst@vectc proiez){#4}}
\def\axesxy#1(#2)(#3){\@ifnextchar({\axesxy@#1(#2)(#3)}{\axesxy@#1(0,0)(#2)(#3)}}
\def\axesxy@#1(#2)(#3)(#4){\axes#1(#2)(#3)(#4){$x$}{$y$}}
\def\axesreim#1(#2)(#3){\@ifnextchar({\axesreim@#1(#2)(#3)}{\axesreim@#1(0,0)(#2)(#3)}}
\def\axesreim@#1(#2)(#3)(#4){\axes#1(#2)(#3)(#4){Re}{Im}}
\def\axes{\@ifnextchar[{\axes@}{\axes@[fillstyle=none]}}
\def\axes@#1(#2)(#3){\@ifnextchar({\axes@@#1(#2)(#3)}{\axes@@#1(0,0)(#2)(#3)}}
\def\axes@@[#1](#2)(#3)(#4)#5#6{{\psset{#1}%
  \psline{->}({#3}|{#2})({#4}|{#2})
  \psline{->}({#2}|{#3})({#2}|{#4})
  \uput{2pt}[270]({#4}|{#2}){#5}
  \uput{2pt}[0]({#2}|{#4}){#6}}}
\def\mediumpoint(#1){%
  \@ifnextchar({\medium@point(#1)}{\medium@point(0,0)(#1)}}
\def\medium@point(#1)(#2)#3{%
  \vnode(0.5*{{#1}+{#2}}){#3}}
\def\incenter(#1)(#2)(#3)#4{%
  \pst@getvect{#1}\pst@vecta%
  \pst@getvect{#2}\pst@vectb%
  \pst@getvect{#3}\pst@vectc%
  \vnode(!\pst@vecta \pst@vectb \pst@vectc incentro){#4}}

\def\centroid(#1)(#2)(#3)#4{%
  \pst@getvect{#1}\pst@vecta%
  \pst@getvect{#2}\pst@vectb%
  \pst@getvect{#3}\pst@vectc%
  \vnode(!\pst@vecta exch \pst@vectb \pst@vectc exch 
  5 3 roll add add 3 div 4 1 roll add add 3 div){#4}}
\def\orthocenter(#1)(#2)(#3)#4{%
  \pst@getvect{#1}\pst@vecta%
  \pst@getvect{#2}\pst@vectb%
  \pst@getvect{#3}\pst@vectc%
  \vnode(!\pst@vecta  \pst@vectb  \pst@vectc  ortoc){#4}}
\def\circumcenter(#1)(#2)(#3)#4{%
  \pst@getvect{#1}\pst@vecta%
  \pst@getvect{#2}\pst@vectb%
  \pst@getvect{#3}\pst@vectc%
  \vnode(!\pst@vecta \pst@vectb \pst@vectc circoc){#4}}
\def\circleCP{\@ifnextchar[{\circle@CP,}{\circle@CP[]}}
\def\circle@CP#1[#2](#3)(#4){%
  \psellipse[showpoints=false#1#2](#3)(`{#3}>{#4}*{1,1})}
\def\rightangle{\@ifnextchar[{\right@angle,}{\right@angle[]}}
\def\right@angle#1[#2](#3)(#4){\@ifnextchar({\right@@angle[#1#2](#3)(#4)}{\right@@angle[#1#2](#3)(0,0)(#4)}}
\def\right@@angle[#1](#2)(#3)(#4){%
  \rput(#3){\psset{unit=2.5mm,dotstyle=*,linewidth=.3pt,showpoints=false,dotsize=2pt#1}%
    \psline(/{{#3}>{#2}})(/{{#3}>{#2}}+/{{#3}>{#4}})(/{{#3}>{#4}})
    \psdots(0.5*{/{{#3}>{#2}}+/{{#3}>{#4}}})}}
\def\arcangle{\@ifnextchar[{%
    \arcangle@i,}{%
    \arcangle@i[]}}
\def\arcangle@i#1[#2]{\@ifnextchar({%
    \arcangle@ii[#1#2]8mm}{%
    \arcangle@ii[#1#2]}}
\def\arcangle@ii#1(#2)#3(#4)(#5){%
  \ifx#3`\arcangle@iii#1(#5)(#4)(#2)
  \else\arcangle@iii#1(#2)(#4)(#5)\fi}
\def\arcangle@iii[#1]#2(#3)(#4)(#5){%
  \psarc[linewidth=.3pt#1](#4){#2}{({#4}>{#3})}{({#4}>{#5})}}
\def\labelangle{\@ifnextchar*{%
    \def\rp@t{\rput*}\labelangle@i}{%
    \def\rp@t{\rput}\labelangle@i*}}
\def\labelangle@i*{\@ifnextchar[{%
    \labelangle@ii}{%
    \labelangle@iii{}}}
\def\labelangle@ii[#1]{\labelangle@iii{[#1]}}
\def\labelangle@iii#1#2(#3)#4(#5)(#6)#7{%
  \ifx#4`\labelangle@iv#2///(#6)(#5)(#3){#7}{#1}
  \else\labelangle@iv#2///(#3)(#5)(#6){#7}{#1}\fi}
\def\labelangle@iv{\@ifnextchar/{%
    \labelangle@v8mm}{%
    \labelangle@v}}
\def\labelangle@v#1/{\@ifnextchar/{%
    \labelangle@vi#1/0.7}{%
   \labelangle@vi#1/}}
\def\labelangle@vi#1/#2/{\@ifnextchar/{%
    \labelangle@vii#1/#2/0.5}{%
    \labelangle@vii#1/#2/}}
\def\labelangle@vii#1/#2/#3/#4(#5)(#6)(#7)#8#9{%
  \pst@getvect{#5}\pst@vecta%
  \pst@getvect{#6}\pst@vectb%
  \pst@getvect{#7}\pst@vectc%
  \pssetlength\pst@dimg{#1}%
  \def\x@dim{#2 \pst@number\pst@dimg \pst@number\psxunit div mul mul }%
  \def\y@dim{#2 \pst@number\pst@dimg \pst@number\psyunit div mul mul }%
  \rp@t({#6}+!\pst@vecta \pst@vectb \pst@vectc 2ang 2 copy lt {exch 360 add exch} if 
  dup 3 1 roll sub #3 mul add dup cos \x@dim exch sin \y@dim){#8}%
  \arcangle#9#1(#5)(#6)(#7)}
\def\@ifgrop#1#2{\@ifnextchar`{#1}{\@ifnextchar({#1}{#2}}}
\def\@ifgropoe#1#2{\@ifgrop{#1}{\@ifnextchar>{#1}{#2}}}
\def\point{\@ifnextchar*{\def\up@t{\uput*}\point@i}{\def\up@t{\uput}\point@i*}}
\def\point@i*{\@ifnextchar[{\point@ii,}{\point@ii[]}}
\def\point@ii#1[#2]{\@ifnextchar({\point@iii[#1#2]2pt}{\point@iii[#1#2]}}
\def\point@iii[#1]#2(#3)#4(#5)(#6)#7{%
  \ifx#4`\point@iii[#1]#2(#6)(#5)(#3){#7}
  \else{\psset{dotstyle=o,dotsize=2.3pt#1}%
  \psdots(#5)\vnode(#3){nnA}\vnode(#5){nnB}\vnode(#6){nnC}%
  \up@t{#2}[!V@nnA V@nnB V@nnC label](#5){#7}}\fi}
\def\vpoint#1({\vpoint@i#1`(}
\def\vpoint@i#1`#2(#3)#4(#5){%
  \ifx#2`\def\next{\vpoint@ii#1(#5)}%
  \else\def\next{\vpoint@ii#1}\fi%
  \next(#3)#4(#5)}
\def\vpoint@ii#1(#2)#3(#4)#5(#6){\point#1(#2)#3(#4)(#6){$#4$}%
  \@ifgrop{\vpoint@ii#1(#4)#5(#6)}{\ifx#5`\point#1(#4)(#6)(#4){$#6$}\fi}}
\def\vlines{\@ifnextchar[{\vlines@i{arrows=c-c,linearc=0.1pt,}}{\vlines@i[arrows=c-c,linearc=0.1pt]}}
\def\vlines@i#1[#2]#3(#4){\vlines@ii[#1#2]#3`(#4)}
\def\vlines@ii[#1]#2`#3(#4)#5(#6){%
  \ifx#3`\def\@next{\vlines@iv{[#1](#4)}{[#1]#2(#6)(#4)}}
  \else\def\@next{\vlines@iii[#1]#2<}\fi
  \@next(#4)#5(#6)}
\def\vlines@iii[#1]#2<#3(#4)#5(#6){%
  \ifx#3<\def\@next{\vlines@iv{[#1](#4)}{[#1]#2(#4)}}
  \else\def\@next{\vlines@iv{[#1]}{[#1]#2(#4)}}\fi
  \@next(#4)#5(#6)}
\def\vlines@iv#1#2(#3)#4(#5){\@ifgropoe{%
    \vlines@iv{#1(#5)}{#2#4(#5)}(#5)}{%
    \ifx#4`\psline#1(#5)\vpoint#2(#5)(#3)
    \else\ifx#4>\psline#1(#5)\vpoint#2(#5)
    \else\psline#1\vpoint#2(#5)\fi\fi}}

\def\vpolygon{\@ifnextchar[{\vpolygon@i{linearc=0.2pt,}}{\vpolygon@i[linearc=0.2pt]}}
\def\vpolygon@i#1[#2]#3(#4){\vpolygon@ii[#1#2]#3`(#4)}
\def\vpolygon@ii[#1]#2`#3(#4)#5(#6){\vpolygon@iii{[#1](#4)}{[#1]#2(#4)}{#3(#4)(#6)}#5(#6)}
  \def\vpolygon@iii#1#2#3#4(#5){\@ifgrop{%
      \vpolygon@iii{#1(#5)}{#2#4(#5)}{#3}}{%
      \pspolygon#1(#5)\vpoint#2#4(#5)#3}}
\def\lineAB{\@ifnextchar({\lineAB@iv}{\@ifnextchar/{\lineAB@v}{\lineAB@i}}}
\def\lineAB@i#1({\lineAB@ii#1/(}
\def\lineAB@ii#1/{\@ifnextchar({\lineAB@iii#1}{\lineAB@vi#1/}}
\def\lineAB@iii#1(#2)(#3)#4/{\lineAB@vi#1/#4/(#2)(#3)#4/}
\def\lineAB@iv(#1)(#2)#3/#4/{\lineAB@vi#4/#3/(#1)(#2)#3/#4/}
\def\lineAB@v/#1(#2)(#3)#4/#5/{\lineAB@vi#5/#1/(#2)(#3)#4/#5/}
\def\lineAB@vi#1/#2/(#3)(#4)#5/#6/{%
  \psline[linestyle=dashed]({#3}+{#1 #2 add}*{{#4}>{#3}})({#4}+{#5 #6 add}*{{#3}>{#4}})%
  \psline({#3}+#2*{{#4}>{#3}})({#4}+#5*{{#3}>{#4}})}
\def\name{\@ifnextchar*{\def\up@t{\uput*}\name@}{\def\up@t{\uput}\name@*}}
\def\name@*#1(#2){\@ifnextchar({\name@@#1//(#2)}{\name@@#1//(0,0)(#2)}}
\def\name@@{\@ifnextchar/{\name@@@1.5pt}{\name@@@}}
\def\name@@@#1/{\@ifnextchar/{\name@@@@#1/0.5}{\name@@@@#1/}}
\def\name@@@@#1/#2/#3(#4)(#5)#6{%
  \up@t{#1}[(90?{#5}>{#4})](#2*{{#4}+{#5}}){#6}}

\def\treD#1#2(#3)#4{%
  \tre@D(#3,){#1}{#2}{#4}}
\def\tre@D(#1,#2,{\@ifnextchar){\tre@@D#1(#2,}{\tre@@D#1,#2(}}
\def\tre@@D#1(#2,)#3#4#5{%
  \pst@getvect{#1}\pst@vecta%
  \vnode(!\pst@vecta exch pop #2+\pst@vecta pop*#4*#3?!1 0){#5}}
\def\uput{\@ifnextchar*{\mathuput@i}{\mathuput@i{}}}
\def\mathuput@i#1{\@ifnextchar[{\mathuput@ii{#1}3pt}{\mathuput@ii{#1}}}
\def\mathuput@ii#1#2[{\@ifnextchar({\mathuput@iii{#1}{#2}[}{\mathuput#1{#2}[}}
\def\mathuput@iii#1#2[(#3)]{%
  \vnode(#3){uputangle}
  \mathuput#1{#2}[(!V@uputangle)]}
\def\mathuput{\def\pst@par{}\pst@ifstar{\@ifnextchar[{\uput@ii}{\uput@i}}}
\makeatother
\newcounter{proposizione}
\newcounter{corollario}
\newcounter{figura}
\newrgbcolor{Bluartico}{0.88 0.95 1.00}
\newenvironment{dimo}{\begin{proof}[{\rm \bf Dimostrazione}]}{\end{proof}}
\def\figura{\normalsize\stepcounter{figura}Figura \thefigura.}
\newtheorem*{main}{Teorema (TT)}
\newtheorem*{prop}{\stepcounter{proposizione}Proposizione \theproposizione{} (P\theproposizione)}
\newtheorem*{corol}{\stepcounter{corollario}Corollario \thecorollario}
\newtheorem*{pro}{\stepcounter{proposizione}Proposizione \theproposizione}
\psset{unit=1mm,labelsep=.6mm}
\newrgbcolor{gray90}{0.90 0.90 0.90}
\newrgbcolor{gray70}{0.70 0.70 0.70}
\def\vvnn#1(#2,#3,#4){\vnode(#2,#3){A_#1}\vnode(#2,#4){B_#1}}
\def\wnn#1(#2,#3,#4){\vnode(5*#2,#3){A_#1}\vnode(5*#2,#4){B_#1}}
\def\f#1#2{\frac{#1}{#2}}
\makeatletter
\def\retta{\@ifnextchar[{\retta@}{\retta@@}}
\def\retta@[#1](#2){{\psset{#1}\retta@@(#2)}}
\def\retta@@(#1){%
  \intersection(5*{0,0|#1})(5*{{1,0}|{#1+{90?#1}}})(55,0)(55,1){@a@}
  \intersection(5*{0,0|#1})(5*{{1,0}|{#1+{90?#1}}})(-2,0)(-2,1){@b@}
  \psline(@a@)(@b@)}
\def\R{\@ifnextchar2{\R@}{\ensuremath{\mathbb{R}}\xspace}}
\def\R@2{\ensuremath{\mathbb{R}^2}\xspace}
\def\axx(#1,#2)(#3,#4){\@ifnextchar({\axx@(#1,#2)(#3,#4)}{\axx@(0,0)(#1,#2)(#3,#4)}}
\def\axx@(#1,#2)(#3,#4)(#5,#6){\axes(#1,#2)(#3,#4)(#5,#6){$x$}{}\uput{2pt}[180](#1,#6){$y$}}
\makeatother
\def\L{\ensuremath{\mathscr{L}(\R^2)}\xspace}
\MathCoor
\begin{document}
\begin{center}
  {\Huge\bf Segmenti paralleli}\\[5mm]
  {\small autore: \bf prof.~Antonio Polo}\\
  {\small Scuola Media Superiore Italiana di Rovigno (Croazia),}\\
  {\small Università ``Juraj Dobrila'' di Pola (Croazia),\\
    email: {\sf toni.rovigno@gmail.com}}\\[5mm]
  {\bf Riassunto}\\[2mm]
  \begin{minipage}{0.8\linewidth}
    In questo articolo cercherò di dare alcune risposte ad un problema matematico che il mio amico Patrizio Frederic, ricercatore di
    statistica presso l'università di Modena, mi ha sottoposto. Dati alcuni segmenti paralleli, vi è almeno una retta che li attraversa
    tutti? Se vi fossero molte rette che risolvono il problema, ve n'è forse una ``migliore'' delle altre? Alla prima domanda risponderò
    esaurientemente, per la seconda mi limiterò a dare alcuni spunti per una futura ricerca.
  \end{minipage}\\[3mm]
  \begin{minipage}{0.75\linewidth}
    {\bf Parole chiave.} Segmento, retta, parallelo, baricentro continuo, baricentro discreto, funzione lineare, operatore, punti, piano
    cartesiano, determinante, coefficiente angolare.
  \end{minipage}\\
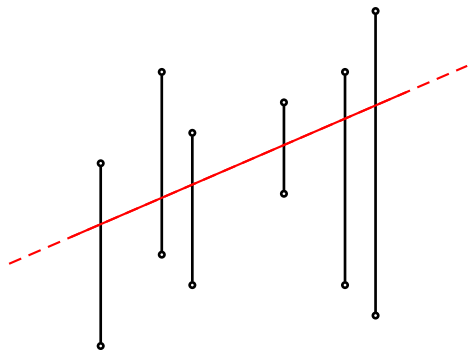
  \begin{pspicture}[unit=0.8mm](-10,4)(64,55)
    \vvnn1(5,5,35)
    \vvnn2(15,20,50)
    \vvnn3(20,15,40)
    \vvnn4(35,30,45)
    \vvnn5(45,15,50)
    \vvnn6(50,10,60)
    \vnode(5*!-1 62 15 div)x
    \vnode(5*!11 28 3 div)y
    \multido{\n=1+1}6{\psline[linewidth=1pt,dotstyle=o,dotsize=3pt]{o-o}(A_\n)(B_\n)}
    \psline[linecolor=red](x+1 12 div*{x>y})(y)
    \psline[linecolor=red,linestyle=dashed](x+1 12 div*{y>x})(x+14 12 div*{x>y})
  \end{pspicture}
\end{center}
\subsection*{Un particolare operatore}
Cominciamo col definire uno strumento che ci permetta di stabilire quali sono le posizioni reciproche di tre punti $A(x_A,y_A)$,
$B(x_B,y_B)$ e $C(x_C,y_C)$ sul piano cartesiano; in particolare ci interessa sapere se i tre punti sono allineati o, in altenativa, su
quale semipiano si trova uno dei tre rispetto alla retta passante per gli altri due. Tale strumento sarà un operatore definito mediante il
determinante di una particolare matrice reale di dimensioni $3\times3$. 

Dati tre punti $A(x_A,y_A)$, $B(x_B,y_B)$ e $C(x_C,y_C)$ sul piano cartesiano definiamo l'operatore $\Phi:(\R^2)^3\to\R$ nel
seguente modo: $$\Phi(A,B,C):=\det\begin{pmatrix}1&1&1\\x_A&x_B&x_C\\y_A&y_B&y_C\end{pmatrix}.$$
Utilizziamo le proprietà di multilinearità del determinante per dimostrare alcune proprietà dell'operatore $\Phi$.
\begin{prop}
  $\Phi$ è invariante per traslazioni.
\end{prop}
\begin{dimo}
  Sia data una traslazione $\mathcal T:\R^2\to\R^2$ di vettore $\vec v=(a,b)$. Abbiamo che
  \begin{gather*}
    \Phi\left(\mathcal T(A),\mathcal T(B),\mathcal T(C)\right)=
    \det\begin{pmatrix}1&1&1\\x_A+a&x_B+a&x_C+a\\y_A+b&y_B+b&y_C+b\end{pmatrix}=\\
    =\det\begin{pmatrix}1&1&1\\x_A&x_B&x_C\\y_A+b&y_B+b&y_C+b\end{pmatrix}+
    a\cdot\overbrace{\det\begin{pmatrix}1&1&1\\1&1&1\\y_A+b&y_B+b&y_C+b\end{pmatrix}}^{=0}=\\
    =\det\begin{pmatrix}1&1&1\\x_A&x_B&x_C\\y_A&y_B&y_C\end{pmatrix}+
    b\cdot\overbrace{\det\begin{pmatrix}1&1&1\\x_A&x_B&x_C\\1&1&1\end{pmatrix}}^{=0}=\Phi(A,B,C).
  \end{gather*}\vskip-1.11cm
\end{dimo}\vskip4.6mm
\begin{prop}
  $\Phi$ è invariante per rotazioni.
\end{prop}
\begin{dimo}
  Sia data una rotazione $\EuScript R:\R^2\to\R^2$ di angolo $\alpha$. Abbiamo che
  \begin{gather*}
    \Phi\left(\EuScript R(A),\EuScript R(B),\EuScript R(C)\right)=
    \det\left[\begin{pmatrix}1&0&0\\0&\cos\alpha&-\sin\alpha\\0&\sin\alpha&\cos\alpha\end{pmatrix}\cdot
      \begin{pmatrix}1&1&1\\x_A&x_B&x_C\\y_A&y_B&y_C\end{pmatrix}\right]=\\
    =\overbrace{\det\begin{pmatrix}1&0&0\\0&\cos\alpha&-\sin\alpha\\0&\sin\alpha&\cos\alpha\end{pmatrix}}^{=1}\cdot
    \det\begin{pmatrix}1&1&1\\x_A&x_B&x_C\\y_A&y_B&y_C\end{pmatrix}=\Phi(A,B,C)
  \end{gather*}\vskip-11.0mm
\end{dimo}
\begin{prop}
  Tre punti $A$, $B$ e $C$ del piano sono allineati se e solo se $\Phi(A,B,C)=0$.
\end{prop}
\begin{dimo}
  Se due dei tre punti coincidessero la proposizione sarebbe ovvia, possiamo pertanto supporre che i tre punti siano distinti. Per la
  proposizione P2 non è restrittivo supporre anche che i punti $A$ e $B$ siano allineati con l'asse delle ascisse, cioè che
  $y_A=y_B$. Abbiamo che 
  $$\Phi(A,B,C)=\det\begin{pmatrix}1&1&1\\x_A&x_B&x_C\\y_A&y_A&y_C\end{pmatrix}=(x_A-x_B)(y_A-y_C);$$
  per l'ipotesi che i punti siano distinti si ha che $x_A\ne x_B$, quindi $\Phi(A,B,C)=0$ se e solo se $y_A=y_C$, cioè se e solo se i punti
  $A$, $B$ e $C$ sono allineati. 
\end{dimo}
\begin{prop}
  Dati tre punti $A$, $B$ e $C$ del piano, distinti e non allineati, sia $\mathscr C$ la circonferenza circoscritta al triangolo
  $ABC$. Allora $\Phi(A,B,C)>0$ se e solo su $\mathscr C$ i punti $A$, $B$ e $C$ sono ordinati in senso antiorario. Viceversa
  $\Phi(A,B,C)<0$ se e solo se su $\mathscr C$ i punti $A$, $B$ e $C$ sono ordinati in senso orario.
\end{prop}
\begin{dimo}
  Per la proposizione P1 non è restrittivo supporre che il circocentro del triangolo $ABC$ coincida con l'origine del piano cartesiano e
  per la P2 non è restrittivo supporre che i punti $A$ e $B$ siano allineati con l'asse delle ordinate, in modo tale che $x_A=x_B$ e che
  $y_A<y_B$. Sia $R$ il raggio della circonferenza circoscritta al triangolo $ABC$, allora per le suddette ipotesi esistono due angoli 
  $\alpha\in\langle 0,180^{\circ}\rangle$ e $\gamma\in[0,360^{\circ}\rangle$, $\gamma\ne\pm\alpha$, tali che $A(R\cos\alpha,-R\sin\alpha)$,
  $B(R\cos\alpha,R\sin\alpha)$ e $C(R\cos\gamma,R\sin\gamma)$ siano le coordinate cartesiane dei punti $A$, $B$ e $C$.
  \begin{center}
    \begin{pspicture}[linewidth=0.5pt,framesep=0pt](0,-29)(121,25)
      \intersection(20;70)(20;-70)(-25,25)(25,25)A
      \multido{\n=32+57,\na=145+180,\nu=1+-2}{2}{\rput(\n,0){%
          \multido{\n=-24+4}{13}{\psline[linecolor=lightgray](\n,-25)(\n,25)\psline[linecolor=lightgray](-25,\n)(25,\n)}%
          \axx(-25,-25)(25,25)%
          \uput*{1pt}[117.5](20;117.5){$\mathscr C$}
          \rput*[l](-23.7,-22.3){$\cos\alpha\ifnum\nu=-1<\else>\fi\cos\gamma$}%
          \rput[l](-25,-27){\figura}%
          \psline[linestyle=dashed](20;-70)(20;70)%
          \psline[linestyle=dashed](20;\na|0,0)(20;\na)%
           \point*(20;\na)(20;\na|0,0)(20;\na){$\scriptstyle{\cos\gamma}$}
           \point*(20;70|{\nu*20;70})(20;70|0,0)(20;70|{\nu*20;70}){$\scriptstyle{\cos\alpha}$}
          \point*(0,0)(20;-70)(0,0){$A$}%
          \point*(0,0)(20;\na)(0,0){$C$}%
          \point*(0,0)(20;70)(0,0){$B$}}}%
      \rput(32,0){\psarc[arcsep=4pt]{->}(0,0){20}{145}{-70}%
        \psarc[arcsep=4pt]{->}(0,0){20}{70}{145}%
        \psarc[arcsep=4pt]{->}(0,0){20}{-70}{70}}%
      \rput(89,0){\psarc[arcsep=4pt]{<-}(0,0){20}{-35}{70}%
        \psarc[arcsep=4pt]{<-}(0,0){20}{-70}{-35}%
        \psarc[arcsep=4pt]{<-}(0,0){20}{70}{-70}}%
    \end{pspicture}
  \end{center}
  Nelle figure 1 e 2 è evidenziato il fatto che se i punti $A$, $B$ e $C$ sono ordinati in senso antiorario, allora $\cos\alpha>\cos\gamma$,
  mentre se i punti $A$, $B$ e $C$ sono ordinati in senso orario, allora $\cos\alpha<\cos\gamma$. Ma calcolando $\Phi(A,B,C)$ abbiamo che
  \begin{gather*}
    \Phi(A,B,C)=\det\begin{pmatrix}1&1&1\\R\cos\alpha&R\cos\alpha&R\cos\gamma\\
      -R\sin\alpha&R\sin\alpha&R\sin\gamma\end{pmatrix}=2R^2\sin\alpha(\cos\alpha-\cos\gamma)
  \end{gather*}
  e siccome $2R^2\sin\alpha>0$, si deduce la tesi:
  \begin{align*}
    &\Phi(A,B,C)>0\quad\Leftrightarrow\quad\cos\alpha>\cos\gamma\quad\Leftrightarrow\quad\text{l'orientamento è antiorario,}\\
    &\Phi(A,B,C)<0\quad\Leftrightarrow\quad\cos\alpha<\cos\gamma\quad\Leftrightarrow\quad\text{l'orientamento è orario.}
  \end{align*}\vskip-6.75mm
\end{dimo}
E per concludere questo paragrafo, sempre riguardo all'operatore $\Phi$, facciamo un'ultima osservazione che ci tornerà utile nel seguito.
\begin{prop}
  Sono dati quattro punti $A$, $B$, $C$ e $D$ nel piano, posizionati in modo tale che $x_B<x_C$ e $x_A=x_D$. Le seguenti
  condizioni sono equivalenti:
  \begin{enumerate}[\rm i)]
  \item $y_A\ge y_D$,
  \item $\Phi(A,B,C)\ge\Phi(D,B,C)$,
  \item $\Phi(B,A,C)\le\Phi(B,D,C)$,
  \item $\Phi(B,C,A)\ge\Phi(B,C,D)$.
  \end{enumerate}
Inoltre in tutti e quattro i casi l'uguaglianza vale se e solo se $A\equiv D$.
\end{prop}
\begin{dimo}
  ($\bm{{\rm i}\Leftrightarrow {\rm ii}}$) Calcoliamo:
  \begin{gather*}
    \Phi(A,B,C)-\Phi(D,B,C)=\det\begin{pmatrix}1&1&1\\x_A&x_B&x_C\\y_A&y_A&y_C\end{pmatrix}-
    \det\begin{pmatrix}1&1&1\\x_A&x_B&x_C\\y_D&y_A&y_C\end{pmatrix}=\\
    =\det\begin{pmatrix}0&1&1\\0&x_B&x_C\\y_A-y_D&y_A&y_C\end{pmatrix}=(y_A-y_D)(x_C-x_B);
  \end{gather*}
  per ipotesi $x_C-x_B>0$, quindi $\Phi(A,B,C)\ge\Phi(D,B,C)\Leftrightarrow y_A\ge y_D$, ed inoltre
  $\Phi(A,B,C)=\Phi(D,B,C)\Leftrightarrow y_A=y_D\Leftrightarrow A\equiv D$.\\[1mm]
  ($\bm{{\rm ii}\Leftrightarrow{\rm iii}\Leftrightarrow{\rm iv}}$) Il determinante, oltre che multilineare, è un operatore alternante,
ovvero cambia di segno per  ogni trasposizione di riga o di colonna, quindi questa equivalenza è ovvia.
\end{dimo}
\subsection*{Lo spazio delle funzioni lineari}
Un altro strumento che utilizzeremo è lo spazio delle funzioni lineari, che chiameremo \L: è una sorta di spazio duale di \R2,
isomorfo a \R2, nel quale non troviamo rappresentati i funzionali lineari (cioè le applicazioni lineari del tipo $f:\R2\to\R$ tale che
$f(x,y)=ax+by$ per qualche $(a,b)\in\R2$),  bensì troviamo rappresentate le funzioni lineari reali, cioè le funzioni $f:\R\to\R$ tali che
$f(x)=kx+l$ per qualche $(k,l)\in\R2$. 

Le funzioni lineari di \L sono in corrispondenza biunivoca con le rette sul piano cartesiano non parallele all'asse delle ordinate, queste
sono in corrispondenza biunivoca con le equazioni del tipo $y=kx+l$ e quindi, tramite la coppia $(k,l)$ sono in corrispondenza biunivoca con
le coppie di \R2; questa corrispondenza ha delle proprietà interessanti evidenziate nella prossima proposizione.
\begin{pro}
  Sia \L lo spazio delle funzioni lineari di \R2. In \L valgono le seguenti proprietà:\\[2pt]
  {\bf(L1)} l'insieme dei fasci di rette sul piano cartesiano è in corrispondenza biunivoca con l'insieme delle rette su \L e, in particolare,
  una retta su \L è parallela alle ordinate se e solo se il fascio a cui è associata è improprio;\\[2pt]
  {\bf(L2)} l'insieme delle rette passanti per i punti di un segmento parallelo all'asse delle ordinate è convesso in \L e, in particolare,
  esso è formato dai punti del piano compresi tra due rette parallele.\\[2pt] 
  {\bf(L3)} l'insieme delle rette passanti per i punti di due segmenti paralleli all'asse delle ordinate, ma con ascisse distinte, è
  limitato e convesso in \L e, in particolare, esso è un parallelogramma il cui baricentro è associato alla retta passante per i punti medi
  dei due segmenti. 
\end{pro}
\begin{dimo}
  {\bf(L1)} Sia dato un generico punto $P(x_P,y_P)\in\R2$, allora il fascio proprio di rette passante per $P$ ha equazione esplicita
  $y=k(x-x_P)+y_P$; al variare di $k\in\R$ questo fascio è rappresentato in \L dal luogo geometrico dei punti le cui coordinate sono
  $(k,-x_Pk+y_P)$. Questo luogo geometrico è una retta la cui equazione è $y=-x_Px+y_P$, quindi non parallela alle ordinate. Se invece
  avessimo un fascio proprio di rette in \R2, la sua equazione esplicita sarebbe del tipo $y=cx+l$, con $l$ parametro e $c$ costante; al
  variare di $l\in\R$ questo fascio è rappresentato in \L dal luogo geometrico dei punti le cui coordinate sono $(c,l)$ e questo luogo
  geometrico è una retta, parallela all'asse delle ordinate, di equazione $x=c$.
  \begin{center}
    \begin{pspicture}[c,linewidth=0.6pt,framesep=0pt](-22,-13)(30,30)%
      \multido{\ny=-8+4,\nx=-20+4}{13} {%
	\ifnum\ny>29{}\else\psline[linecolor=lightgray](-22,\ny)(30,\ny)\fi%
	\psline[linecolor=lightgray](\nx,-10)(\nx,30)}%
      \rput[l](-22,-12){\figura}%
      \axx(-22,-10)(30,30)%
      \psline[linewidth=2pt,linecolor=blue](4*1,1)(4*1,4)%
      \point*(4*1,4)(4*1,1)(4*1,4){$B$}%
      \point*(4*1,1)(4*1,4)(4*1,1){$A$}%
      \uput*{2pt}[180](0,4){$\scriptstyle1$}%
      \uput*{2pt}[270](4,0){$\scriptstyle1$}%
      \intersection(0,20)(8,8)(0,-10)(1,-10)A%
      \intersection(0,20)(8,8)(0,30)(1,30)B%
      \psline[linestyle=dashed](A)(B)
      \psline(A+0.1*{A>B})(B+0.1*{B>A})
      \pcline[linestyle=none,offset=-5pt](B)(B+0.25*{B>A})\lput*{:U}{$\scriptscriptstyle r_1:\,y=-\frac32x+5$}
      \intersection(-16,0)(-8,4)(30,0)(30,1)A%
      \intersection(-16,0)(-8,4)(-22,0)(-22,1)B%
      \psline[linestyle=dashed](A)(B)
      \psline(A+0.1*{A>B})(B+0.1*{B>A})
      \pcline[linestyle=none,offset=5pt](B+0.75*{B>A})(A)\lput*{:U}{$\scriptscriptstyle r_4:\,y=\frac12x+2$}
      \intersection(28,-8)(-20,24)(30,0)(30,1)A%
      \intersection(28,-8)(-20,24)(-22,0)(-22,1)B%
      \psline[linestyle=dashed](A)(B)
      \psline(A+0.1*{A>B})(B+0.1*{B>A})
      \pcline[linestyle=none,offset=5pt](B)(B+0.25*{B>A})\lput*{:U}{$\scriptscriptstyle r_2:\,y=-\frac23x+\frac83$}
      \intersection(0,12)(4,12)(30,0)(30,1)A%
      \intersection(0,12)(4,12)(-22,0)(-22,1)B%
      \psline[linestyle=dashed](A)(B)
      \psline(A+0.1*{A>B})(B+0.1*{B>A})
      \pcline[linestyle=none,offset=5pt](B)(B+0.15*{B>A})\lput*{:U}{$\scriptscriptstyle r_3:\,y=3$}%
    \end{pspicture}\hfill$\Longleftrightarrow$\hfill
    \begin{pspicture}[c,linewidth=0.6pt,framesep=0pt](-26,-13)(26,30)%
      \intersection(0,4)(4,0)(0,-10)(1,-10)A%
      \intersection(0,4)(4,0)(0,30)(1,30)B%
      \intersection(0,16)(16,0)(0,-10)(1,-10)C%
      \intersection(0,16)(16,0)(0,30)(1,30)D%
      \pspolygon[linestyle=none,fillstyle=solid,fillcolor=Bluartico](A)(B)(D)(C)
      \multido{\ny=-8+4,\nx=-24+4}{13}{%
	\ifnum\ny>29{}\else\psline[linecolor=lightgray](-26,\ny)(26,\ny)\fi%
	\psline[linecolor=lightgray](\nx,-10)(\nx,30)}
      \axx(-26,-10)(26,30)%
      \uput*{2pt}[202.5](0,4){$\scriptstyle1$}%
      \uput*{2pt}[22.5](0,16){$\scriptstyle4$}%
      \uput*{2pt}[247.5](4,0){$\scriptstyle1$}%
      \psline[linestyle=dashed](A)(B)
      \psline(A+0.1*{A>B})(B+0.1*{B>A})
      \psline[linestyle=dashed](C)(D)
      \psline(C+0.1*{C>D})(D+0.1*{D>C})
      {\psset{fillcolor=Bluartico}
	\uput*[140](4*0,3){$\scriptstyle r_3$}
	\uput*[5](2,8){$\scriptstyle r_4$}
	\uput*[175](-6,20){$\scriptstyle r_1$}
	\uput*[265](4 3 div*-2,8){$\scriptstyle r_2$}}
      \psdots[dotstyle=o](4*0,3)(2,8)(-6,20)(4 3 div*-2,8)
    \end{pspicture}
  \end{center}
  {\bf(L2)} Sia $\overline{AB}$ un segmento parallelo all'asse delle ordinate di estremi $A(a,b)$ e $B(a,c)$ (figura 3); un generico punto
  $P_t$ di questo segmento ha coordinate del tipo $P_t\big(a,b+t(c-b)\big)$ per qualche $t\in[0,1]$. Per quanto visto in (L1), il fascio
  proprio di rette passante per $P_t$ è rappresentato in \L dalle rette di equazione $y=-ax+b+t(c-b)$; al variare di $t\in[0,1]$ tali rette 
  rappresentano l'insieme di tutte le rette parallele comprese fra le rette $y=-ax+b$ e $y=-ax+c$ e formano un insieme convesso in
  \L.\\[2pt] 
  {\bf(L3)} Per quanto visto in (L2) è evidente che l'insieme in questione è un parallelogramma. Per quanto riguarda la seconda parte della
  tesi, siano $A(a,b)$, $B(a,c)$, $C(d,e)$ e $D(d,f)$ i vertici dei due segmenti e siano $E\big(a,\frac{b+c}2\big)$ e
  $F\big(d,\frac{e+f}2\big)$ i punti medi di essi; l'equazione della retta $EF$ sarà 
  $$y=\frac{b+c-e-f}{2(a-d)}(x-d)+\frac{e+f}2=\frac{b+c-e-f}{2(a-d)}x+\frac{ae+af-bd-cd}{2(a-d)}\,.$$
  Le equazioni delle rette $AC$, $AD$, $BC$ e $BD$ sono:
  \begin{align*}
    AC:&\,y=\frac{b-e}{a-d}x+\frac{ae-bd}{a-d}\quad,&AD:&\,y=\frac{b-f}{a-d}x+\frac{af-bd}{a-d}\;,\\
    BC:&\,y=\frac{c-e}{a-d}x+\frac{ae-cd}{a-d}\quad\text e&BD:&\,y=\frac{c-f}{a-d}x+\frac{af-cd}{a-d}\;.
  \end{align*}
  Queste quattro rette, in \L, rappresentano i vertici del parallelogramma corrispondente a tutte le rette passanti per $AB$ e $CD$; il
  baricentro di tale parallelogramma è il punto medio di una sua diagonale, ovvero il punto
  $$P\left(\frac{\frac{b-e}{a-d}+\frac{c-f}{a-d}}{2},\frac{\frac{ae-bd}{a-d}+\frac{af-cd}{a-d}}{2}\right)\;\Rightarrow\;
  P\left(\frac{b-e+c-f}{2(a-d)},\frac{ae-bd+af-cd}{2(a-d)}\right)$$ le cui coordinate corrispondono effettivamente ai coefficienti della
  retta $EF$.
\end{dimo}
\subsection*{Il problema}
Il problema che vogliamo risolvere è il seguente: dato un certo numero finito $n\ge2$ di segmenti chiusi, paralleli, ma a due a due non
allineati, si tratta di trovare delle condizioni necessarie e sufficienti affinché vi sia almeno una retta che li attraversi tutti.\\[2mm]
{\bf Notazioni} (figura 4).
\begin{itemize}
  \item Poniamo i segmenti su un piano cartesiano (figura 4), in modo tale che risultino essere paralleli all'asse delle ordinate, e
    chiamiamoli $L_1,L_2,\dots,L_n$ ordinandoli crescentemente secondo le ascisse.
  \item Siano $A_i(x_i,a_i)$ e $B_i(x_i,b_i)$ gli estremi dell'$i$-esimo segmento, con $a_i\le b_i$ per ogni $i\in\{1,2,\dots,n\}$ e
    $x_1<x_2<\cdots<x_n$. 
  \item Fra tutte le rette passanti per i vertici dei segmenti ne fissiamo due particolari: la retta $A_sB_t$ il cui
    coefficiente angolare  $k_{st}$ è il minimo tra i coefficienti angolari di tutte le rette del tipo $A_iB_j$ con $1\le i<j\le n$ e la
    retta $B_uA_v$ il cui  coefficiente angolare $k'_{uv}$ è il massimo tra i coefficienti angolari di tutte le rette del tipo $B_iA_j$
    con $1\le i<j\le n$.
\end{itemize}
\begin{center}\small
  \vvnn1(10,5,35)
  \vvnn2(25,20,50)
  \vvnn3(35,15,40)
  \vvnn4(55,30,45)
  \vvnn5(65,15,50)
  \vvnn6(100,10,60)
  \begin{pspicture}[framesep=0pt,c](-5,-6)(110,63)
    \multido{\n=5+5}{12}{\psline[linecolor=lightgray](-5,\n)(110,\n)}
    \multido{\n=5+5}{21}{\psline[linecolor=lightgray](\n,-3)(\n,63)}
    \axx(-5,-3)(110,63)
    \multido{\n=1+1}6{\psline[linewidth=2pt,linecolor=blue](A_\n)(B_\n)}
    \multido{\n=1+1}{5}{\name*(A_\n)(B_\n){$L_\n$}\point*(A_\n)(B_\n)(A_\n){$B_\n(x_\n,b_\n)$}
      \point*(B_\n)(A_\n)(B_\n){$A_\n(x_\n,a_\n)$}}
    \name*(A_6)(B_6){$L_n$}
    \point*(A_6)(B_6)(A_6){$B_n(x_n,b_n)$}\point*(B_6)(A_6)(B_6){$A_n(x_n,a_n)$}
    \rput*(82.5,35){\Huge$\cdots$}
    \rput[l](-5,-5){\figura}
    \name(A_1)(B_1){$L_1$}
  \end{pspicture}
\end{center}
\begin{main}
  Con le notazioni precedenti e considerando il caso non banale $n\ge3$, abbiamo che i seguenti fatti sono equivalenti:
  \begin{enumerate}[\rm i)]
  \item esiste almeno una retta che interseca tutti i segmenti $L_1,L_2,\dots,L_n$;
  \item $\Phi(A_i,B_j,A_k)\le0\le\Phi(B_i,A_j,B_k)$ per ogni $i,j,k$ tali che $1\le i<j<k\le n$;
  \item la retta $A_sB_t$ interseca tutti i segmenti $L_1,L_2,\dots,L_n$;
  \item la retta $B_uA_v$ interseca tutti i segmenti $L_1,L_2,\dots,L_n$.
  \end{enumerate}
\end{main}
\begin{proof}[{\rm \bf Dimostrazione}]
  ($\bm{{\rm i}\Rightarrow{\rm ii}}$) Fissiamo tre segmenti $L_i$, $L_j$ e $L_k$, con $1\le i<j<k\le n$; per ipotesi esiste una retta che 
  li attraversa,  quindi possiamo supporre che esistano tre punti allineati
  $P_i(x_i,y_i)\in L_i$, $P_j(x_j,y_j)\in L_j$ e $P_k(x_k,y_k)\in L_k$
  tali che  $a_p\le y_p\le b_p$ per ogni $p\in\{i,j,k\}$. Per la P3 sappiamo che $\Phi(P_i,P_j,P_k)=0$ e per la P5 abbiamo che\vskip-8mm
  \begin{gather*}
    \Phi(A_i,B_j,A_k)\le\Phi(P_i,B_j,A_k)\le\Phi(P_i,P_j,A_k)\le\overbrace{\Phi(P_i,P_j,P_k)}^{=0}\le\\
    \le\Phi(B_i,P_j,P_k)\le\Phi(B_i,A_j,P_k)\le\Phi(B_i,A_j,P_k)\,,
  \end{gather*}
  quindi questa prima implicazione è dimostrata.\\[1mm]
  ($\bm{{\rm ii}\Rightarrow{\rm iii}}$) Per ogni $i\in\{1,2,\ldots,n\}$ sia $P_i(x_i,y_i)$ il punto di intersezione tra la retta $A_sB_t$ e
  la retta $x=x_i$. Ci sono quattro casi da analizzare.\\[.5mm]
  $1^{\text o}$ caso: $1\le i<s$. Per ipotesi $\Phi(B_i,A_s,B_t)\ge 0$ e per la P3 $\Phi(P_i,A_s,B_t)=0$, quindi per la P5 abbiamo che
  $b_i\ge y_i$.\\[.5mm]
  $2^{\text o}$ caso: $1\le i<t$. Sappiamo che $k_{st}\le k_{it}$ e che il punto $P_i$ appartiene alla retta $A_sB_t$, quindi abbiamo che
  $$k_{st}\le k_{it}\quad\Leftrightarrow\quad\frac{y_i-b_t}{x_i-x_t}\le\frac{a_i-b_t}{x_i-x_t}
  \quad\Leftrightarrow\quad y_i-b_t\ge a_i-b_t\quad\Leftrightarrow\quad y_i\ge a_i\,.$$
  $3^{\text o}$ caso: $s<i\le n$. Sappiamo che $k_{st}\le k_{si}$ e che il punto $P_i$ appartiene alla retta $A_sB_t$, quindi abbiamo che
  $$k_{st}\le k_{si}\quad\Leftrightarrow\quad\frac{a_s-y_i}{x_s-x_i}\le\frac{a_s-b_i}{x_s-x_i}
  \quad\Leftrightarrow\quad a_s-y_i\ge a_s-b_i\quad\Leftrightarrow\quad y_i\le b_i\,.$$
  $4^{\text o}$ caso: $t<i\le n$. Per ipotesi $\Phi(A_s,B_t,A_i)\le 0$ e per la P3 $\Phi(A_s,B_t,P_i)=0$, quindi per la P5 abbiamo che
  $a_i\le y_i$.\\[.5mm]
  Aggiungendo anche le due disuguaglianze banali $b_s\ge y_s=a_s$ e $a_t\le y_t=b_t$, nei quattro casi sono dimostrate tutte le
  disuguaglianze del tipo $a_i\le y_i\le b_i$ per ogni $i\in\{1,2,\ldots n\}$, quindi la retta $A_sB_t$ interseca tutti gli $n$
  segmenti.\\[1mm]
  ($\bm{{\rm ii}\Rightarrow{\rm iv}}$) Dimostrazione del tutto analoga alla precedente.\\[1mm]
  ($\bm{{\rm iii}\Rightarrow{\rm i}}$), ($\bm{{\rm iv}\Rightarrow{\rm i}}$). Queste implicazioni sono ovvie.
\end{proof}
\begin{corol}
  I seguenti fatti sono equivalenti:
  \begin{enumerate}[\rm i)]
  \item esiste un'unica retta che interseca tutti i segmenti $L_1,L_2,\dots,L_n$;
  \item le rette $A_sB_t$ e $B_uA_v$ coincidono e intersecano i segmenti $L_1,L_2,\dots,L_n$;
  \item $\Phi(A_i,B_j,A_k)\le0\le\Phi(B_i,A_j,B_k)$ per ogni $i,j,k$ tali che $1\le i<j<k\le n$ e, se non siamo in presenza del caso limite
  in cui due dei segmenti degenerano in un punto, allora almeno in un caso si ha l'uguaglianza
  $$\Phi(A_i,B_j,A_k)=0\quad\text{oppure}\quad\Phi(B_i,A_j,B_k)=0\,.$$
  \end{enumerate}
\end{corol}
\begin{proof}[{\rm \bf Dimostrazione}]
  ($\bm{{\rm i}\Rightarrow{\rm ii}}$) Ovvio, per l'equivalenza ${\rm i}\Leftrightarrow{\rm iii}\Leftrightarrow{\rm iv}$ di TT.\\[1mm]
  ($\bm{{\rm ii}\Rightarrow{\rm iii}}$) TT ci assicura che $\Phi(A_i,B_j,A_k)\le0\le\Phi(B_i,A_j,B_k)$ per ogni $i,j,k$ tali che
  $1\le i<j<k\le n$. I quattro punti $A_s$, $B_t$, $B_u$ e $A_v$ sono, per ipotesi, allineati; supponendo che non vi siano due segmenti che
  degenerano in un punto, allora non può contemporaneamente accadere che $A_s\equiv B_u$ e $B_t\equiv A_v$, quindi siamo certi che, oltre
  alle disuguaglianze già note $x_s<x_t$ e $x_u<x_v$, si verifica anche almeno una fra  queste: 
  $$x_s<x_u\,,\quad x_u<x_s\,,\quad x_t<x_v\,,\quad x_v<x_t\,.$$ 
  In ciascuno dei quattro casi abbiamo che
  \begin{align*}
   \qquad\qquad x_s&<x_u\!\!\!\!\!&\vee&&\!\!\!\!\!x_u&<x_v&\Rightarrow&&\Phi(A_s,B_u,A_v)&=0\,,\qquad\qquad\\
   \qquad\qquad x_u&<x_s\!\!\!\!\!&\vee&&\!\!\!\!\!x_s&<x_t&\Rightarrow&&\Phi(B_u,A_s,B_t)&=0\,,\qquad\qquad\\
   \qquad\qquad x_t&<x_v\!\!\!\!\!&\vee&&\!\!\!\!\!x_s&<x_t&\Rightarrow&&\Phi(A_s,B_t,A_v)&=0\,,\qquad\qquad\\
   \qquad\qquad x_v&<x_t\!\!\!\!\!&\vee&&\!\!\!\!\!x_u&<x_v&\Rightarrow&&\Phi(B_u,A_v,B_t)&=0\,,\qquad\qquad
  \end{align*}
  e questo completa la dimostrazione.\\[1mm]
  ($\bm{{\rm iii}\Rightarrow{\rm i}}$) TT ci assicura che almeno una retta esiste; l'unicità sarebbe ovvia se due dei segmenti
  fossero puntiformi, quindi supponiamo che $\Phi(A_i,B_j,A_k)=0$ per qualche $i,j,k$ tali che $1\le i<j<k\le n$ (la dimostrazione è
  analoga se supponiamo che $\Phi(B_i,A_j,B_k)=0$). Una retta che attraversa i  segmenti $L_i$, $L_j$ e $L_k$ passerà per i tre punti
  allineati $P_i\in L_i$, $P_j\in L_j$ e $P_k\in L_k$, quindi per la P5 avremo che
  $$0=\Phi(A_i,B_j,A_k)\le\Phi(P_i,B_j,A_k)\le\Phi(P_i,P_j,A_k)\le\Phi(P_i,P_j,P_k)=0\,.$$
  Poiché tutte le disuguaglianze risultano essere delle uguaglianze allora, sempre per la P5, abbiamo che $P_i\equiv A_i$, 
  $P_j\equiv B_j$ e $P_k\equiv A_k$ e la retta non può che essere unica.
\end{proof}
\begin{corol}
  Tra tutte le rette che attraversano i segmenti $L_1,L_2,\ldots,L_n$ (supponendo che ce ne siano) le rette $A_sB_t$ e $B_uA_v$ sono
  quelle con coefficiente angolare rispettivamente massimo e minimo.
\end{corol}
\begin{proof}[{\rm \bf Dimostrazione}]
  Consideriamo una retta $r\!:\;y=kx+l$ di coefficiente angolare $k>k_{st}$, passante per i punti $P_s(x_s,y_s)\in L_s$ e $P_t(x_t,y_t)$ e
  dimostriamo che $P_t\notin L_t$. Calcoliamo:
  \begin{gather*}
    k_{st}<k\quad\Rightarrow\quad \frac{a_s-b_t}{x_s-x_t}<\frac{y_s-y_t}{x_s-x_t}\quad\Rightarrow\\
    \Rightarrow\quad a_s-b_t>y_s-y_t\quad\Rightarrow\quad y_t-b_t>y_s-a_s\ge0.
  \end{gather*}
  Siccome $y_t>b_t$ allora la retta $r$ non attraversa il segmento $L_t$ e $k_{st}$ è quindi il coefficiente angolare massimo fra quelli
  delle rette che attraversano tutti i segmenti. Analogamente $k'_{uv}$ sarà il minimo dei coefficienti angolari.
\end{proof}
\subsection*{La retta speciale}
A questo punto siamo in grado di suggerire un paio di algoritmi che ci permettono di trovare una retta che, fra le tante, possa
attraversare nel ``migliore'' dei modi tutti i segmenti.\\[2mm]
{\bf Algoritmo 1.} Una retta speciale... ma non troppo.
\begin{enumerate}
 \item Dati gli $n$ segmenti, troviamo la retta $r=A_sB_t$.
 \item Verifichiamo che sia $\Phi(B_i,A_s,B_t)\ge0$ per ogni $i\in\{1,2,\ldots,s-1\}$.
 \item Verifichiamo che sia $\Phi(A_s,B_t,A_j)\le0$ per ogni $j\in\{t+1,t+2,\ldots,n\}$.
 \item Se le verifiche ai punti $2$ e $3$ hanno avuto esito negativo non vi è motivo di proseguire, altrimenti cerchiamo anche la retta
  $p=B_uA_v$. Per TT le rette $r$ e $p$ passano per tutti i segmenti $L_1,L_2,\ldots,L_n$.
 \item Prendiamo la retta $s_1$ che, in \L, corrisponde al punto medio fra $r$ e $p$. Per L2 l'insieme delle rette passanti per tutti i
  segmenti $L_1,L_2,\ldots,L_n$ è convesso, quindi anche $s$ passa per tali segmenti. 
\end{enumerate}
\begin{center}\psset{unit=1mm}
  \small%
  \vvnn1(5,5,35)%
  \vvnn2(15,20,50)%
  \vvnn3(20,15,40)%
  \vvnn4(35,30,45)%
  \vvnn5(45,15,50)%
  \vvnn6(50,10,60)%
  \begin{pspicture}[c,framesep=0pt](-4,-6)(56,66)
    \pspolygon[linestyle=none,fillstyle=solid,fillcolor=Bluartico](A_1)(A_2)(A_3)(A_4)(A_5)(A_6)(B_6)(B_5)(B_4)(B_3)(B_2)(B_1)
    \multido{\n=5+5}{12}{\psline[linecolor=lightgray](-2,\n)(55,\n)}
    \multido{\n=5+5}{10}{\psline[linecolor=lightgray](\n,-2)(\n,65)}
    \axx(-2,-2)(55,65)
    \multido{\n=1+1}6{\psline[linewidth=1pt,linecolor=blue](A_\n)(B_\n)}
    \intersection(A_2)(B_5)(55,0)(55,1)X
    \intersection(A_2)(B_5)(-2,0)(-2,1)W
    \psline[linestyle=dashed,linewidth=0.6pt](X)(W)
    \pcline[linestyle=none,offset=3pt](X+0.08*{X>W})(X)\lput*{:U}{$\bm r$}
    \intersection(A_4)(B_1)(55,0)(55,1)Y
    \intersection(A_4)(B_1)(-2,0)(-2,1)Z
    \psline[linestyle=dashed,linewidth=0.6pt](Y)(Z)
    \pcline[linestyle=none,offset=4pt](Y)(Y+0.08*{Y>Z})\lput*{:D}{$\bm p$}
    \psline[linecolor=red](0.5*{X+Y})(0.5*{W+Z})
    \pcline[linestyle=none,offset=3pt](0.5*{X+Y}+0.08*{{0.5*{X+Y}}>{0.5*{W+Z}}})(0.5*{X+Y})\lput*{:U}{\color{red}$\bm{s_1}$}
    \vpoint(B_1)(A_1)(A_2)(A_3)(A_4)(A_5)
    \vpoint(A_5)(A_6)(B_6)(B_5)
    \vpoint1pt(B_6)(B_5)(B_4)
    \vpoint(B_5)(B_4)(B_3)(B_2)(B_1)
    \vpoint0pt(A_4)(A_5)(A_6)
    \vpoint2pt(B_2)(B_1)(Z)
    \multido{\n=1+1}{10}{\uput*{2pt}[d](5*\n,0){\n}}
    \multido{\n=1+1}{12}{\uput*{2pt}[l](5*0,\n){\n}}
    \uput[r](-4,-5){\figura}
  \end{pspicture}\hfill%
  \vvnn3(20,30,40)%
  \vvnn4(35,20,45)%
  \begin{pspicture}[c,framesep=0pt](-3,-6)(56,66)
    \pspolygon[linestyle=none,fillstyle=solid,fillcolor=Bluartico](A_1)(A_2)(A_3)(A_4)(A_5)(A_6)(B_6)(B_5)(B_4)(B_3)(B_2)(B_1)
    \multido{\n=5+5}{12}{\psline[linecolor=lightgray](-2,\n)(55,\n)}
    \multido{\n=5+5}{10}{\psline[linecolor=lightgray](\n,-2)(\n,65)}
    \axx(-2,-2)(55,65)
    \multido{\n=1+1}6{\psline[linewidth=1pt,linecolor=blue](A_\n)(B_\n)}
    \intersection(A_3)(B_5)(55,0)(55,1)X
    \intersection(A_3)(B_5)(-2,0)(-2,1)W
    \psline[linestyle=dashed,linewidth=0.6pt](X)(W)
    \pcline[linestyle=none,offset=3pt](X+0.08*{X>W})(X)\lput*{:U}{$\bm r$}
    \intersection(A_3)(B_1)(55,0)(55,1)Y
    \intersection(A_3)(B_1)(-2,0)(-2,1)Z
    \psline[linestyle=dashed,linewidth=0.6pt](Y)(Z)
    \pcline[linestyle=none,offset=4pt](Y)(Y+0.08*{Y>Z})\lput*{:D}{$\bm p$}
    \psline[linecolor=red](0.5*{X+Y})(0.5*{W+Z})
    \pcline[linestyle=none,offset=3pt](0.5*{X+Y}+0.08*{{0.5*{X+Y}}>{0.5*{W+Z}}})(0.5*{X+Y})\lput*{:U}{\color{red}$\bm{s_1}$}
    \vpoint(B_1)(A_1)(A_2)(A_3)(A_4)
    \vpoint(A_5)(A_6)(B_6)(B_5)
    \vpoint1pt(B_6)(B_5)(B_4)
    \vpoint(B_5)(B_4)(B_3)(B_2)(B_1)
    \vpoint0pt(A_3)(A_4)(A_5)(A_6)
    \vpoint2pt(B_2)(B_1)(Z)
    \multido{\n=1+1}{10}{\uput*{2pt}[d](5*\n,0){\n}}
    \multido{\n=1+1}{12}{\uput*{2pt}[l](5*0,\n){\n}}
    \uput[r](-4,-5){\figura}
  \end{pspicture}
\end{center}
In figura 5 abbiamo che
$$r\!:\,y=x+1\quad\wedge\quad p\!:\,y=-\frac16x+\frac{43}6\quad\Rightarrow\quad s_1\!:\,y=\frac5{12}x+\frac{49}{12}\,;$$
in figura 6 abbiamo che
$$r\!:\,y=\frac 45x+\frac{14}5\quad\wedge\quad p\!:\,y=-\frac13+\frac{22}3\quad\Rightarrow\quad
s_1\!:\,y=\frac{7}{30}x+\frac{76}{15}\,.$$
Ma questo metodo non sempre ci dà una ``buona'' retta; come è evidenziato nella figura 6, a volte la retta $s_1$ che si ottiene non
sembrerebbe essere una delle migliori, in quanto è una retta che passa per uno dei vertici.\\[2mm]
{\bf Algoritmo 2.} Una retta un po' più speciale.
\begin{enumerate}
 \item Eseguiamo i punti 1, 2, 3 e 4 dell'algoritmo 1.
 \item Tra tutte le rette $A_iA_j$, distinte da $r$ e da $p$, con $1\le i<j\le n$, prendiamo solo quelle che attraversano tutti i segmenti
  $L_1,L_2,\ldots,L_n$. 
 \item Tra tutte le rette $B_iB_j$, distinte da $r$ e da $p$, con $1\le i<j\le n$, prendiamo solo quelle che attraversano tutti i segmenti
  $L_1,L_2,\ldots,L_n$. 
 \item Per L2 l'insieme delle rette trovate ai punti 5 e 6, assieme alle rette $r$ e $p$, in \L sono i vertici di un poligono
  convesso. Prendiamo come retta $s_2$ quella corrispondente al baricentro discreto di tale poligono.
\end{enumerate}
Rivediamo l'esempio della figura 5.
\begin{center}\psset{unit=1mm}
  \small%
  \vvnn1(5,5,35)%
  \vvnn2(15,20,50)%
  \vvnn3(20,15,40)%
  \vvnn4(35,30,45)%
  \vvnn5(45,15,50)%
  \vvnn6(50,10,60)%
  \begin{pspicture}[c,framesep=0pt](-4,-6)(56,66)
    \pspolygon[linestyle=none,fillstyle=solid,fillcolor=Bluartico](A_1)(A_2)(A_3)(A_4)(A_5)(A_6)(B_6)(B_5)(B_4)(B_3)(B_2)(B_1)
    \multido{\n=5+5}{12}{\psline[linecolor=lightgray](-2,\n)(55,\n)}
    \multido{\n=5+5}{10}{\psline[linecolor=lightgray](\n,-2)(\n,65)}
    \axx(-2,-2)(55,65)
    \multido{\n=1+1}6{\psline[linewidth=1pt,linecolor=blue](A_\n)(B_\n)}
    \intersection(A_2)(B_5)(55,0)(55,1)X
    \intersection(A_2)(B_5)(-2,0)(-2,1)W
    \psline[linestyle=dashed,linewidth=0.6pt](X)(W)
    \pcline[linestyle=none,offset=3pt](X+0.08*{X>W})(X)\lput*{:U}{$\bm r$}
    \intersection(A_4)(B_1)(55,0)(55,1)Y
    \intersection(A_4)(B_1)(-2,0)(-2,1)Q
    \psline[linestyle=dashed,linewidth=0.6pt](Y)(Q)
    \pcline[linestyle=none,offset=4pt](Y)(Y+0.08*{Y>Q})\lput*{:D}{$\bm p$}
    \intersection(A_2)(A_4)(55,0)(55,1)Y
    \intersection(A_2)(A_4)(-2,0)(-2,1)Z
    \psline[linestyle=dashed,linewidth=0.6pt](Y)(Z)
    \intersection(B_4)(B_1)(55,0)(55,1)Y
    \intersection(B_4)(B_1)(-2,0)(-2,1)Z
    \psline[linestyle=dashed,linewidth=0.6pt](Y)(Z)
    \intersection(B_4)(B_5)(55,0)(55,1)Y
    \intersection(B_4)(B_5)(-2,0)(-2,1)Z
    \psline[linestyle=dashed,linewidth=0.6pt](Y)(Z)
    \intersection(5*1,5)(5*137 30 div*0,1)(55,0)(55,1)a
    \intersection(5*1,5)(5*137 30 div*0,1)(-2,0)(-2,1)b
    \psline[linecolor=red](a)(b)
    \pcline[linestyle=none,offset=3pt](a+0.08*{a>b})(a)\lput*{:U}{\color{red}$\bm{s_2}$}
    \vpoint(B_1)(A_1)(A_2)(A_3)(A_4)(A_5)
    \vpoint(A_5)(A_6)(B_6)(B_5)
    \vpoint1pt(B_6)(B_5)(B_4)
    \vpoint(B_5)(B_4)(B_3)(B_2)(B_1)
    \vpoint0pt(A_4)(A_5)(A_6)
    \vpoint2pt(B_2)(B_1)(Q)
    \multido{\n=1+1}{10}{\uput*{2pt}[d](5*\n,0){\n}}
    \multido{\n=1+1}{12}{\uput*{2pt}[l](5*0,\n){\n}}
    \uput[r](-4,-5){\figura}
  \end{pspicture}\hfill%
  \begin{pspicture}[c,framesep=0pt](-10,-8)(47,64)
    \vnode(21,15){P_1}
    \vnode(-7,43){P_2}
    \vnode(14,40){P_3}
    \vnode(21,33){P_4}
    \vnode(42,06){P_5}
    \pspolygon[fillstyle=solid, fillcolor=Bluartico](P_1)(P_2)(P_3)(P_4)(P_5)
    {\psset{linecolor=lightgray}
      \psline(-7,-3)(-7,63)
      \psline(14,-3)(14,63)
      \psline(21,-3)(21,63)
      \psline(42,-3)(42,63)
      \psline(-10,15)(47,15)
      \psline(-10,6)(47,6)
      \psline(-10,33)(47,33)
      \psline(-10,40)(47,40)
      \psline(-10,43)(47,43)}
    \axx(-10,-3)(47,63)
    \vpoint*2pt(P_4)(P_3)(P_2)(P_1)
    \vpoint2pt(P_2)(P_1)(P_5)(P_4)(P_3)
    \point[linecolor=red,fillcolor=red](P_5)(1 5 div*{P_1+P_2+P_3+P_4+P_5})(P_5){\color{red}G}
    \uput*[d](-7,0){$-\frac16$}
    \uput*[d](14,0){$\frac13$}
    \uput*[d](21,0){$\frac12$}
    \uput*[d](42,0){$1$}
    \uput*[l](0,6){$1$}
    \uput*[l](0,15){$\frac52$}
    \uput*[l](0,33){$\frac{11}2$}
    \uput[ur](8,48){$\frac{43}6$}
    \uput[dl](-6,39){$\frac{20}3$}
    \psline[linewidth=0.3pt,linestyle=dotted,dotsep=1pt]{->}(9.3,50.1)(0.3,43.3)
    \psline[linewidth=0.3pt,linestyle=dotted,dotsep=1pt]{->}(-7.3,37.05)(-0.3,39.7)
    \uput{0pt}[r](-10,-7){\figura}
  \end{pspicture}
\end{center}
In figura 7 abbiamo rappresentato i segmenti e in figura 8 abbiamo rappresentato in \L il poligono convesso formato delle rette che
attraversano tutti i segmenti; il punto $G$ è il baricentro discreto del poligono a cui è associata la retta $s_2$ che sembrerebbe
attraversare i segmenti in modo migliore rispetto alla retta $s_1$ della figura 5. Riassumendo abbiamo
che:\renewcommand\minalignsep{2pt}
\begin{align*}
  A_4B_1\!:\;y=&-\frac16x+\frac{43}6&&\Rightarrow&&\!P_2\left(-\frac16,\frac{43}6\right)\,,&
  A_2A_4\!:\;y=&\frac12x+\frac52\!&&\Rightarrow&&\!P_1\left(\frac12,\frac52\right)\,,\\
  B_1B_4\!:\;y=&\frac13x+\frac{20}3\!&&\Rightarrow&&\!P_3\left(\frac13,\frac{20}3\right)\,,&
  B_5A_2\!:\;y=&x+1\!&&\Rightarrow&&P_5\left(1,1\right)\,,\\
  B_4B_5\!:\;y=&\frac12x+\frac{11}2&&\Rightarrow&&\!P_4\left(\frac12,\frac{11}2\right)\,;
\end{align*}
$$x_G=\frac{-\frac16+\frac13+\frac12+\frac12+1}{5}=\frac{13}{30}\;,\quad
y_G=\frac{\frac{43}6+\frac{20}3+\frac{11}2+\frac52+1}{5}=\frac{137}{30}$$
ed infine $s_2\!:\,y=\dfrac{13}{30}x+\dfrac{137}{30}$. \newpage\noindent
Rivediamo anche l'esempio della figura 6.
\begin{center}\psset{unit=1mm}
  \small%
  \vvnn1(5,5,35)%
  \vvnn2(15,20,50)%
  \vvnn3(20,30,40)%
  \vvnn4(35,20,45)%
  \vvnn5(45,15,50)%
  \vvnn6(50,10,60)%
  \begin{pspicture}[c,framesep=0pt](-4,-6)(56,66)
    \pspolygon[linestyle=none,fillstyle=solid,fillcolor=Bluartico](A_1)(A_2)(A_3)(A_4)(A_5)(A_6)(B_6)(B_5)(B_4)(B_3)(B_2)(B_1)
    \multido{\n=5+5}{12}{\psline[linecolor=lightgray](-2,\n)(55,\n)}
    \multido{\n=5+5}{10}{\psline[linecolor=lightgray](\n,-2)(\n,65)}
    \axx(-2,-2)(55,65)
    \multido{\n=1+1}6{\psline[linewidth=1pt,linecolor=blue](A_\n)(B_\n)}
    \intersection(A_3)(B_5)(55,0)(55,1)X
    \intersection(A_3)(B_5)(-2,0)(-2,1)W
    \psline[linestyle=dashed,linewidth=0.6pt](X)(W)
    \pcline[linestyle=none,offset=3pt](X+0.08*{X>W})(X)\lput*{:U}{$\bm r$}
    \intersection(A_3)(B_1)(55,0)(55,1)Y
    \intersection(A_3)(B_1)(-2,0)(-2,1)Q
    \psline[linestyle=dashed,linewidth=0.6pt](Y)(Q)
    \pcline[linestyle=none,offset=4pt](Y)(Y+0.08*{Y>Q})\lput*{:D}{$\bm p$}
    \intersection(B_4)(B_1)(55,0)(55,1)Y
    \intersection(B_4)(B_1)(-2,0)(-2,1)Z
    \psline[linestyle=dashed,linewidth=0.6pt](Y)(Z)
    \intersection(B_4)(B_5)(55,0)(55,1)Y
    \intersection(B_4)(B_5)(-2,0)(-2,1)Z
    \psline[linestyle=dashed,linewidth=0.6pt](Y)(Z)
    \intersection(5*-11,2)(5*223 40 div*0,1)(55,0)(55,1)a
    \intersection(5*-11,2)(5*223 40 div*0,1)(-2,0)(-2,1)b
    \psline[linecolor=red](a)(b)
    \pcline[linestyle=none,offset=3pt](a+0.08*{a>b})(a)\lput*{:U}{\color{red}$\bm{s_2}$}
    \vpoint(B_1)(A_1)(A_2)(A_3)(A_4)(A_5)
    \vpoint(A_5)(A_6)(B_6)(B_5)
    \vpoint1pt(B_6)(B_5)(B_4)
    \vpoint(B_5)(B_4)(B_3)(B_2)(B_1)
    \vpoint0pt(A_4)(A_5)(A_6)
    \vpoint2pt(B_2)(B_1)(Q)
    \multido{\n=1+1}{10}{\uput*{2pt}[d](5*\n,0){\n}}
    \multido{\n=1+1}{12}{\uput*{2pt}[l](5*0,\n){\n}}
    \uput[r](-4,-5){\figura}
  \end{pspicture}\hfill%
  \begin{pspicture}[c,framesep=0pt](-19,-8)(38,64)
    \vnode(!4 42 mul 5 div 14 6 mul 5 div){P_1}
    \vnode(-14,44){P_2}
    \vnode(14,40){P_3}
    \vnode(21,33){P_4}
    \pspolygon[fillstyle=solid, fillcolor=Bluartico,linearc=0.05](P_1)(P_2)(P_3)(P_4)
    {\psset{linecolor=lightgray}
      \psline(-14,-3)(-14,63)
      \psline(14,-3)(14,63)
      \psline(21,-3)(21,63)
      \psline(P_1|0,-3)(P_1|0,63)
      \psline(-19,0|P_1)(38,0|P_1)
      \psline(-19,33)(38,33)
      \psline(-19,40)(38,40)
      \psline(-19,44)(38,44)}
    \axx(-19,-3)(38,63)
    \vpoint*2pt(P_2)(P_1)(P_4)(P_3)(P_2)(P_1)
    \point[linecolor=red,fillcolor=red](P_1)(1 4 div*{P_1+P_2+P_3+P_4})(P_1){\color{red}G}
    \uput*[d](-14,0){$-\frac13$}
    \uput*[d](14,0){$\frac13$}
    \uput*[d](21,0){$\frac12$}
    \uput*[d](P_1|0,0){$\frac45$}
    \uput*[l](0,0|P_1){$\frac{14}5$}
    \uput*[l](0,33){$\frac{11}2$}
    \uput[ur](8,48){$\frac{22}3$}
    \uput[dl](-6,39){$\frac{20}3$}
    \psline[linewidth=0.3pt,linestyle=dotted,dotsep=1pt]{->}(9.3,50.1)(0.3,44.3)
    \psline[linewidth=0.3pt,linestyle=dotted,dotsep=1pt]{->}(-7.3,37.05)(-0.3,39.7)
    \uput{0pt}[r](-19,-7){\figura}
  \end{pspicture}
\end{center}
Anche in questo caso la retta $s_2$ attraversa in maniera decisamente ``migliore'' tutti i segmenti, rispetto alla retta $s_1$ della figura
6. Riassumendo in questo caso abbiamo che:
\begin{align*} 
  B_5A_3\!:\;y=&\frac45x+\frac{14}5\!&&\Rightarrow&&P_1\left(\frac45,\frac{14}5\right)\,,&
  A_3B_1\!:\;y=&-\frac13x+\frac{22}3&&\Rightarrow&&\!P_2\left(-\frac13,\frac{22}3\right)\,,\\[3mm]
  B_1B_4\!:\;y=&\frac13x+\frac{20}3\!&&\Rightarrow&&\!P_3\left(\frac13,\frac{20}3\right)\,,&
  B_4B_5\!:\;y=&\frac12x+\frac{11}2&&\Rightarrow&&\!P_4\left(\frac12,\frac{11}2\right)\,;
\end{align*}
$$x_G=\frac{\frac45-\frac13+\frac13+\frac12}4=\frac{13}{40}\;,\quad
y_G=\frac{\frac{14}5+\frac{22}3+\frac{20}3+\frac{11}2}4=\frac{223}{40}$$
ed infine $s_2\!:\,y=\dfrac{13}{40}x+\dfrac{223}{40}$.\\[2mm]
{\bf Algoritmo 3.} Una retta ancora più speciale.
\begin{enumerate}
  \item Eseguiamo i punti 1, 2 e 3 dell'algoritmo 2.
  \item Fatte le stesse osservazioni del punto 4 dell'algoritmo 2, come retta $s_3$ prendiamo il baricentro continuo del poligono che
    troviamo in \L.
\end{enumerate}
Rivediamo nuovamente i due esempi precedenti.
\begin{center}\psset{unit=1mm}
  \small%
  \vvnn1(5,5,35)%
  \vvnn2(15,20,50)%
  \vvnn3(20,15,40)%
  \vvnn4(35,30,45)%
  \vvnn5(45,15,50)%
  \vvnn6(50,10,60)%
  \begin{pspicture}[c,framesep=0pt](-4,-6)(56,66)
    \pspolygon[linestyle=none,fillstyle=solid,fillcolor=Bluartico](A_1)(A_2)(A_3)(A_4)(A_5)(A_6)(B_6)(B_5)(B_4)(B_3)(B_2)(B_1)
    \multido{\n=5+5}{12}{\psline[linecolor=lightgray](-2,\n)(55,\n)}
    \multido{\n=5+5}{10}{\psline[linecolor=lightgray](\n,-2)(\n,65)}
    \axx(-2,-2)(55,65)
    \multido{\n=1+1}6{\psline[linewidth=1pt,linecolor=blue](A_\n)(B_\n)}
    \intersection(5*1,5)(5*137 30 div*0,1)(55,0)(55,1)a
    \intersection(5*1,5)(5*137 30 div*0,1)(-2,0)(-2,1)b
    \psline[linestyle=dashed,linecolor=green](a)(b)
    \pcline[linestyle=none,offset=4pt](a+0.08*{a>b})(a)\lput*{:U}{\color{green}$\bm{s_2}$}
    \intersection(A_2)(B_5)(55,0)(55,1)X
    \intersection(A_2)(B_5)(-2,0)(-2,1)W
    \intersection(A_4)(B_1)(55,0)(55,1)Y
    \intersection(A_4)(B_1)(-2,0)(-2,1)Z
    \psline[linestyle=dashed,linecolor=cyan](0.5*{X+Y})(0.5*{W+Z})
    \pcline[linestyle=none,offset=4pt](0.5*{X+Y})(0.5*{X+Y}+0.08*{{0.5*{X+Y}}>{0.5*{W+Z}}})\lput*{:D}{\color{cyan}$\bm{s_1}$}
    \intersection(!0 107 24 div 5 mul)(!5 117 24 div 5 mul)(55,0)(55,1)X
    \intersection(!0 107 24 div 5 mul)(!5 117 24 div 5 mul)(-2,0)(-2,1)W
    \psline[linecolor=red](X)(W)
    \pcline[linestyle=none,offset=1pt](X)(X+0.08*{X>W})\lput*{:D}{\color{red}$\bm{s_3}$}
    \vpoint(B_1)(A_1)(A_2)(A_3)(A_4)(A_5)
    \vpoint(A_5)(A_6)(B_6)(B_5)
    \vpoint1pt(B_6)(B_5)(B_4)
    \vpoint(B_5)(B_4)(B_3)(B_2)(B_1)(A_1)
    \vpoint0pt(A_4)(A_5)(A_6)
    \multido{\n=1+1}{10}{\uput*{2pt}[d](5*\n,0){\n}}
    \multido{\n=1+1}{12}{\uput*{2pt}[l](5*0,\n){\n}}
    \uput[r](-4,-5){\figura}
  \end{pspicture}\hfill%
  \vvnn1(5,5,35)%
  \vvnn2(15,20,50)%
  \vvnn3(20,30,40)%
  \vvnn4(35,20,45)%
  \vvnn5(45,15,50)%
  \vvnn6(50,10,60)%
  \begin{pspicture}[c,framesep=0pt](-4,-6)(56,66)
    \pspolygon[linestyle=none,fillstyle=solid,fillcolor=Bluartico](A_1)(A_2)(A_3)(A_4)(A_5)(A_6)(B_6)(B_5)(B_4)(B_3)(B_2)(B_1)
    \multido{\n=5+5}{12}{\psline[linecolor=lightgray](-2,\n)(55,\n)}
    \multido{\n=5+5}{10}{\psline[linecolor=lightgray](\n,-2)(\n,65)}
    \axx(-2,-2)(55,65)
    \multido{\n=1+1}6{\psline[linewidth=1pt,linecolor=blue](A_\n)(B_\n)}
    \intersection(5*-11,2)(5*223 40 div*0,1)(55,0)(55,1)a
    \intersection(5*-11,2)(5*223 40 div*0,1)(-2,0)(-2,1)b
    \psline[linestyle=dashed,linecolor=green](a)(b)
    \pcline[linestyle=none,offset=4pt](a+0.08*{a>b})(a)\lput*{:U}{\color{green}$\bm{s_2}$}
    \intersection(A_3)(B_5)(55,0)(55,1)X
    \intersection(A_3)(B_5)(-2,0)(-2,1)W
    \intersection(A_3)(B_1)(55,0)(55,1)Y
    \intersection(A_3)(B_1)(-2,0)(-2,1)Z
    \psline[linestyle=dashed,linecolor=cyan](0.5*{X+Y})(0.5*{W+Z})
    \pcline[linestyle=none,offset=4pt](0.5*{X+Y})(0.5*{X+Y}+0.08*{{0.5*{X+Y}}>{0.5*{W+Z}}})\lput*{:D}{\color{cyan}$\bm{s_1}$}
    \intersection(!0 11873 2130 div 5 mul)(!5 752111 127800 div 5 mul)(55,0)(55,1)X
    \intersection(!0 11873 2130 div 5 mul)(!5 752111 127800 div 5 mul)(-2,0)(-2,1)W
    \psline[linecolor=red](X)(W)
    \pcline[linestyle=none,offset=1pt](X)(X+0.08*{X>W})\lput*{:D}{\color{red}$\bm{s_3}$}
    \vpoint(A_5)(A_6)(B_6)(B_5)
    \vpoint1pt(B_6)(B_5)(B_4)
    \vpoint(B_5)(B_4)(B_3)(B_2)(B_1)(A_1)(A_2)(A_3)(A_4)(A_5)
    \vpoint0pt(A_4)(A_5)(A_6)
    \multido{\n=1+1}{10}{\uput*{2pt}[d](5*\n,0){\n}}
    \multido{\n=1+1}{12}{\uput*{2pt}[l](5*0,\n){\n}}
    \uput[r](-4,-5){\figura}
  \end{pspicture}
\end{center}
Osserviamo il poligono in figura 8:
\begin{itemize}
 \item per L1 le equazioni delle rette $P_1P_2$, $P_2P_3$, $P_3P_4$, $P_4P_5$ e $P_5P_1$ sono: 
  \begin{gather*}
    P_1P_2\!:\,y=-7x+6\;,\quad P_2P_3\!:\,y=-x+7\;,\quad P_3P_4\!:\,y=-7x+9\;,\\
    P_4P_5\!:\,y=-9x+10\;,\quad P_5P_1\!:\,y=-3x+4\;;
  \end{gather*}
 \item la superficie del poligono misura:
  \begin{align*}
    A&=\int_SdS=\int_{-\frac16}^{\frac13}\int^{-x+7}_{-7x+6}dydx+\int_{\frac13}^{\frac12}\!\!\int^{-7x+9}_{-7x+6}dydx+
    \int_{\frac12}^{1}\!\!\int^{-9x+10}_{-3x+4}dydx=\\
    &=\int_{-\frac16}^{\frac13}[y]^{-x+7}_{-7x+6}dx+\int_{\frac13}^{\frac12}[y]^{-7x+9}_{-7x+6}dx+
    \int_{\frac12}^{1}[y]^{-9x+10}_{-3x+4}dx=\\
    &=\int_{-\frac16}^{\frac13}(6x+1)dx+\int_{\frac13}^{\frac12}3dx+\int_{\frac12}^{1}(6-6x)dx=\\
    &=\big[3x^2+x\big]_{-\frac16}^{\frac13}+\big[3x\big]_{\frac13}^{\frac12}+
      \big[6x-3x^2\big]_{\frac12}^{1}=\frac34+\frac12+\frac34=\bm2\,;
  \end{align*}
 \item il baricentro continuo $M(x_M,y_M)$ lo troveremo calcolando:
  \begin{align*}
    x_M&=\frac1A\int_SxdS=\\
    &=\frac12\left(\int_{-\frac16}^{\frac13}\int^{-x+7}_{-7x+6}\!\!\!xdydx+\int_{\frac13}^{\frac12}\!\!\!\int^{-7x+9}_{-7x+6}\!\!xdydx+
    \int_{\frac12}^{1}\!\!\int^{-9x+10}_{-3x+4}\!\!\!xdydx\right)=\\
    &=\frac12\left(\int_{-\frac16}^{\frac13}x[y]^{-x+7}_{-7x+6}dx+\int_{\frac13}^{\frac12}x[y]^{-7x+9}_{-7x+6}dx+
    \int_{\frac12}^{1}x[y]^{-9x+10}_{-3x+4}dx\right)=\\
    &=\int_{-\frac16}^{\frac13}\left(3x^2+\frac12x\right)dx+\int_{\frac13}^{\frac12}\frac32xdx+\int_{\frac12}^{1}(3x-3x^2)dx=\\
    &=\left[x^3+\frac14x^2\right]_{-\frac16}^{\frac13}\!\!+\left[\frac34x^2\right]_{\frac13}^{\frac12}\!\!+
      \left[\frac32x^2-x^3\right]_{\frac12}^{1}=\frac1{16}+\frac5{48}+\frac14=\bm{\frac5{12}}\,,
  \end{align*}
  \begin{align*}
    y_M&=\frac1A\int_SydS=\\
    &=\frac12\left(\int_{-\frac16}^{\frac13}\int^{-x+7}_{-7x+6}\!\!\!ydydx+\int_{\frac13}^{\frac12}\!\!\!\int^{-7x+9}_{-7x+6}\!\!ydydx+
    \int_{\frac12}^{1}\!\!\int^{-9x+10}_{-3x+4}\!\!\!ydydx\right)=\\
    &=\frac14\left(\int_{-\frac16}^{\frac13}\left[y^2\right]^{-x+7}_{-7x+6}dx+\int_{\frac13}^{\frac12}\left[y^2\right]^{-7x+9}_{-7x+6}dx+
    \int_{\frac12}^{1}\left[y^2\right]^{-9x+10}_{-3x+4}dx\right)=\\
    &=\!\!\!\int_{\!-\frac16}^{\frac13}\!\!\!\!\left(\!\frac{13}4\!+\!\frac{35}2x-12x^2\!\!\right)\!dx\!+\!\!\!
    \int_{\frac13}^{\frac12}\!\!\!\left(\!\frac{45}4\!-\!\frac{21}2x\!\!\right)\!dx\!+\!\!\!
    \int_{\frac12}^{1}\!\!\!\!\left(21-39x+18x^2\right)\!dx\!=\\
    &=\left[\frac{13}4x+\frac{35}4x^2-4x^3\right]_{-\frac16}^{\frac13}\!\!\!+
    \left[\frac{45}4x-\frac{21}4x^2\right]_{\frac13}^{\frac12}\!\!+
    \left[21x-\frac{39}2x^2+6x^3\right]_{\frac12}^{1}\!=\\
    &=\frac{35}{16}+\frac{55}{48}+\frac98=\bm{\frac{107}{24}}\;;
  \end{align*}
 \item al baricentro continuo corrisponde la retta $s_3\!:\,y=\dfrac{5}{12}x+\dfrac{107}{24}$ (figura 11).
\end{itemize}
Osserviamo ora il poligono in figura 10:
\begin{itemize}
  \item per L1 le equazioni delle rette $P_1P_2$, $P_2P_3$, $P_3P_4$ e $P_4P_5$ sono: 
  \begin{align*}
    P_1P_2\!:\,y&=-4x+6\;,&P_2P_3\!:\,y&=-x+7\;,\;,\\
    P_3P_4\!:\,y&=-7x+9\;,&P_4P_1\!:\,y&=-9x+10\;;
  \end{align*}
 \item la superficie del poligono misura:
  \begin{align*}
    A&=\int_SdS=\int_{-\frac13}^{\frac13}\int^{-x+7}_{-4x+6}\!dydx+\int_{\frac13}^{\frac12}\!\!\int^{-7x+9}_{-4x+6}\!dydx+
    \int_{\frac12}^{\frac45}\!\!\int^{-9x+10}_{-4x+6}\!dydx=\\
    &=\ldots=\frac23+\frac7{24}+\frac9{40}=\bm{\frac{71}{60}}\,;
  \end{align*}
 \item il baricentro continuo $M(x_M,y_M)$ lo troveremo calcolando:
  \begin{align*}
    x_M&=\frac1A\int_SxdS=\\
    &=\frac{60}{71}\left(\int_{-\frac13}^{\frac13}\!\int^{-x+7}_{-4x+6}\!\!xdydx+
    \int_{\frac13}^{\frac12}\!\!\int^{-7x+9}_{-4x+6}\!\!\!xdydx+\int_{\frac12}^{\frac45}\!\!\int^{-9x+10}_{-4x+6}\!\!xdydx\!\right)=\\
    &=\ldots=\frac{60}{71}\left(\frac{2}{27}+\frac{13}{108}+\frac{27}{200}\right)=\bm{\frac{39731}{127800}}\,,
  \end{align*}
  \begin{align*}
    y_M&=\frac1A\int_SydS=\\
    &=\frac{60}{71}\left(\int_{-\frac13}^{\frac13}\!\int^{-x+7}_{-4x+6}\!\!ydydx+
    \int_{\frac13}^{\frac12}\!\!\int^{-7x+9}_{-4x+6}\!\!\!ydydx+\int_{\frac12}^{\frac45}\!\!\int^{-9x+10}_{-4x+6}\!\!ydydx\!\right)=\\
    &=\ldots=\frac{60}{71}\left(\frac{112}{27}+\frac{659}{432}+\frac{369}{400}\right)=\bm{\frac{11873}{2130}}
  \end{align*}
 \item al baricentro continuo corrisponde la retta $s_3\!:\,y=\dfrac{39731}{127800}x+\dfrac{11873}{2130}$ (figura 12).
\end{itemize}
Nei due esempi sembrerebbe di capire che la retta che troviamo con l'algoritmo 1 sia poco valida, mentre il baricentro discreto e
quello continuo ci danno delle rette non molto dissimili tra loro. Se questo fosse sempre vero, allora si potrebbe considerare più
opportuno utilizzare l'algoritmo 2, se non altro per evitare la enorme mole di calcoli che si devono fare per trovare la retta con
l'algoritmo 3. Ma questo risultato è casuale, dovuto a quei pochi esempi visti, ed il prossimo esempio ne è la conferma. Propongo tale
esempio in conclusione dell'articolo e senza calcoli.
\begin{center}\psset{unit=1mm}
  \small%
  \wnn1(1,1,11)%
  \wnn2(2,2,12)%
  \wnn3(4,3,12)%
  \wnn4(6,2,11)%
  \wnn5(8,2,12)%
  \wnn6(9,3,13)%
  \begin{pspicture}[c,framesep=0pt](-4,-6)(56,71)\small
    {\psset{linestyle=none,fillstyle=solid,fillcolor=Bluartico}
      \pspolygon(A_1)(A_2)(A_3)(A_4)(A_5)(A_6)(B_6)(B_5)(B_4)(B_3)(B_2)(B_1)}
    \multido{\n=5+5}{13}{\psline[linecolor=lightgray](-2,\n)(55,\n)}
    \multido{\n=5+5}{10}{\psline[linecolor=lightgray](\n,-2)(\n,70)}
    \axx(-2,-2)(55,70)
    \multido{\n=1+1}6{\psline[linewidth=1pt,linecolor=blue](A_\n)(B_\n)}
    \retta[linecolor=cyan](!1 4 div 23 4 div)
    \pcline[linestyle=none,offset=4pt](@a@+0.08*{@a@>@b@})(@a@)\lput*{:U}{\color{cyan}$\bm{s_1}$}
    \retta[linecolor=green](!7 16 div 77 16 div)
    \pcline[linestyle=none,offset=4pt](@a@+0.08*{@a@>@b@})(@a@)\lput*{:U}{\color{green}$\bm{s_2}$}
    \retta[linecolor=red](!11 46 div 267 46 div)
    \pcline[linestyle=none,offset=4pt](@a@)(@a@+0.08*{@a@>@b@})\lput*{:D}{\color{red}$\bm{s_3}$}
    \vpoint*(A_1)(A_2)(A_3)(A_4)(A_5)(A_6)(B_6)(B_5)(B_4)(B_3)(B_2)(B_1)(A_1)(A_2)
    \multido{\n=1+1}{10}{\uput*{2pt}[d](5*\n,0){\n}}
    \multido{\n=1+1}{13}{\uput*{2pt}[l](5*0,\n){\n}}
    \uput[r](-4,-5){\figura}
  \end{pspicture}\hfill%
  \begin{pspicture}[c,framesep=0pt](-23,-10)(35,67)
    \vnode(21,0){P_1}
    \vnode(10.5,5){P_2}
    \vnode(0,15){P_3}
    \vnode(-21,60){P_4}
    \vnode(0,55){P_5}
    \vnode(10.5,40){P_6}
    \vnode(21,20){P_7}
    \vnode(31.5,-2.5){P_8}
    \pspolygon[fillstyle=solid, fillcolor=Bluartico,linearc=0.05](P_1)(P_2)(P_3)(P_4)(P_5)(P_6)(P_7)(P_8)
    {\psset{linecolor=lightgray}
      \psline(P_1|0,-6)(P_1|0,66)
      \psline(P_2|0,-6)(P_2|0,66)
      \psline(P_4|0,-6)(P_4|0,66)
      \psline(P_8|0,-6)(P_8|0,66)
      \psline(-23,0|P_2)(35,0|P_2)
      \psline(-23,0|P_8)(35,0|P_8)
      \psline(-23,0|P_3)(35,0|P_3)
      \psline(-23,0|P_4)(35,0|P_4)
      \psline(-23,0|P_5)(35,0|P_5)
      \psline(-23,0|P_6)(35,0|P_6)
      \psline(-23,0|P_7)(35,0|P_7)}
    \axx(-23,-6)(35,66)
    \vpoint*2pt(P_2)(P_1)(P_8)(P_7)(P_6)(P_5)(P_4)(P_3)(P_2)(P_1)
    \uput*[u](P_8|0,0){$\frac32$}
    \uput*[d](P_6|0,0){$\frac12$}
    \uput*[d](P_4|0,0){$-1$}
    \uput*[l](0,0|P_4){$12$}
    {\psset{fillcolor=Bluartico}
      \uput*[u](P_7|0,0){$1$}
      \uput*[15](0,0|P_3){$3$}
      \uput*[r](0,0|P_7){$4$}
      \uput*[dl](0,0|P_5){$11$}
      \uput*[l](0,0|P_6){$8$}}
    \uput*[l](0,0|P_2){$1$}
    \uput*[l](0,0|P_8){$-\frac12$}
    \uput*[45](0,0){$O$}
    \uput{0pt}[r](-23,-9){\figura}
  \end{pspicture}
\end{center}
\begin{align*}
  A_1A_2\!:\,&y=x&&P_1\left(1,0\right)		&B_1B_4\!:\,&y=11&&P_5\left(0,11\right)\\
  A_2A_3\!:\,&y=\f12x+1&&P_2\left(\f12,1\right)&B_4B_5\!:\,&y=\f12x+8&&P_6\left(\f12,8\right)\\
  A_3A_6\!:\,&y=3&&P_3\left(0,3\right)		&B_5B_6\!:\,&y=x+4&&P_7\left(1,4\right)\\
  A_6B_1\!:\,&y=-x+12&&P_4\left(-1,12\right)	&B_6A_1\!:\,&y=\f32x-\f12&&P_8\left(\f32,-\f12\right)
\end{align*}
 $$s_1\!:\,y=\f14x+\f{23}4\;,\quad s_2\!:\,y=\f7{16}x+\f{77}{16}\;,\quad s_3\!:\,y=\f{11}{46}x+\frac{267}{46}\;.$$
Come vediamo in questo caso la retta $s_3$ è molto vicina alla retta $s_1$, mentre la retta $s_2$ sembrerebbe un po' troppo inclinata, ma
credo che a questo punto sia ora che il mio amico statistico si riprenda in mano il problema e,
in virtù di questi risultati, lo analizzi con i suoi strumenti.
\end{document}
